\documentclass[sts,
	reqno,
]{imsart}

\usepackage{amsthm, amssymb, amsmath}
\usepackage{bbm}
\usepackage[toc,page]{appendix}
\usepackage{mathtools, thmtools}
\usepackage{euscript}
\usepackage{color}
\usepackage{tikz}
\usetikzlibrary{arrows}
\usepackage[utf8]{inputenc}
\usepackage[shortlabels]{enumitem}
\usepackage{url}
\usepackage{etex}
\usepackage{alltt}
\usepackage{memhfixc}
\usepackage{bbold}%
\usepackage{cases}
\usepackage[noadjust]{cite}
\definecolor{labelkey}{rgb}{0.0, 0.8, 0.3}
\usepackage[top=2.5cm,bottom=2.5cm,left=2.5cm,right=2.5cm,includeheadfoot,asymmetric]{geometry}
\usepackage{algorithm}
\usepackage{xfrac}
\usepackage{wrapfig}

\usepackage{style}
\usepackage{nicematrix}

\numberwithin{equation}{section}
\usepackage{hyperref}
\hypersetup{
  colorlinks = true,
  urlcolor = blue, 
  linkcolor = blue,
  citecolor = blue}
 
\usepackage{cleveref}

\declaretheorem[name=Definition, style=definition]{definition}

\declaretheorem[name=Lemma]{lemma}
\declaretheorem[name=Proposition]{proposition}

\declaretheorem[name=Remark, style=remark]{remark}
\declaretheorem[name=Theorem]{theorem}

\renewcommand{\abstractname}{Abstract}

\usepackage{xcolor,colortbl}
\usepackage{comment}

\def\TV{\sf{TV}}

\def\one{\mathbbm{1}}

\def\E{\mathbb{E}}
\def\eqdef{\triangleq}

\newcommand{\poi}{\operatorname{Poi}}

\renewcommand{\sf}{\mathsf}
\renewcommand{\cal}{\mathcal}
\newcommand{\bb}{\mathbb}

\newcommand{\ber}{\operatorname{Ber}}

\renewcommand{\it}[1]{\textit{#1}}

\renewcommand{\KL}{\sf{KL}}

\newcommand\subsetsim{\mathrel{%
  \ooalign{\raise0.2ex\hbox{$\subset$}\cr\hidewidth\raise-0.8ex\hbox{\scalebox{0.9}{$\sim$}}\hidewidth\cr}}}

\newcommand{\LF}{\mathsf{LFHT}}
\newcommand{\GoF}{\mathsf{GoF}}
\newcommand{\TS}{\mathsf{TS}}
\renewcommand{\P}{\mathbb{P}}
\renewcommand{\d}{\mathrm{d}}
\renewcommand{\D}{\d}
\renewcommand{\sep}{\operatorname{sep}}

\newtheorem{corollary}[theorem]{Corollary}

\newcommand{\cat}{\sf{CAT}}

\begin{document}
% \maketitle
% \tableofcontents

\begin{frontmatter}

	\title{Minimax optimal testing \\by classification}
	\runtitle{Minimax optimal testing by classification}
	\author{Patrik Róbert Gerber \hfill prgerber@mit.edu \\ 
	Yanjun Han \hfill yanjunhan@nyu.edu \\ 
        Yury Polyanskiy \hfill yp@mit.edu \\}

	\address{{Department of Mathematics} \\
		{Massachusetts Institute of Technology}\\
		{77 Massachusetts Avenue,}\\
		{Cambridge, MA 02139, USA}}
	\address{
		 {Institute for Data, Systems, and Society} \\
		{Massachusetts Institute of Technology}\\
		{50 Ames St,}\\
		{Cambridge, MA 02142, USA}
	}
        \address{
		 {Department of Electrical Engineering and Computer Science} \\
		{Massachusetts Institute of Technology}\\
		{32 Vassar St,}\\
		{Cambridge, MA 02142, USA}
	}
	
	\runauthor{Gerber, Han and Polyanskiy}
	\begin{abstract}  
	\small{This paper considers an ML inspired approach to hypothesis testing known as classifier/classification-accuracy testing ($\cat$). In $\cat$, one first trains a classifier by feeding it labeled synthetic samples generated by the null and alternative distributions, which is then used to predict labels of the actual data samples. This method is widely used in practice when the null and alternative are only specified via simulators (as in many scientific experiments). 

We study goodness-of-fit, two-sample ($\TS$) and likelihood-free hypothesis testing ($\LF$), and show that $\cat$ achieves (near-)minimax optimal sample complexity in both the dependence on the total-variation ($\TV$) separation $\epsilon$ and the probability of error $\delta$ in a variety of non-parametric settings, including discrete distributions, $d$-dimensional distributions with a smooth density, and the Gaussian sequence model. In particular, we close the high probability sample complexity of $\LF$ for each class. As another highlight, we recover the minimax optimal complexity of $\TS$ over discrete distributions, which was recently established by \cite{diakonikolas2021optimal}. The corresponding $\cat$ simply compares empirical frequencies in the first half of the data, and rejects the null when the classification accuracy on the second half is better than random. }
\end{abstract}
\end{frontmatter}

\tableofcontents

\section{Introduction}\label{sec:intro}

The rapid development of machine learning over the past three decades has had a profound impact on many areas of science and technology. It has replaced or enhanced traditional statistical procedures and automated feature extraction and prediction where in the past human experts had to intervene manually. One example is the technique that has become known as `classification accuracy testing` (CAT). The idea, first explicitly described in \cite{friedman2004multivariate}, is extremely simple. Consider the setting of two-sample testing: suppose the statistician has samples $X$ and $Y$ of size $n$ from two distributions $\bb P_\sf{X}$ and $\bb P_\mathsf{Y}$ respectively on some space $\cal X$, and wishes to test the hypotheses
\begin{equation}\label{eqn:intro two sample}\tag{$\sf{TS}$}
    H_0: \bb P_\sf{X} = \bb P_\sf{Y} \qquad\text{versus}\qquad H_1 : \bb P_\sf{X} \neq \bb P_\sf{Y}. 
\end{equation}
The statistician has many classical methods at their disposal such as the Kolmogorov-Smirnov or the
Wilcoxon – Mann – Whitney test. Friedman's idea was to use machine learning as a powerful tool to
summarize the data and subsequently apply a classical two-sample test to the transformed data.
More concretely, the proposal is to train a binary classifier $\cal C:\cal X \to \{0,1\}$ on the
labeled data $\cup_{i=1}^n \{(X_i,0), (Y_i,1)\}$ and compare the samples $\cal C(X_1), \dots, \cal
C(X_n)$ and $\cal C(Y_1), \dots, \cal C(Y_n)$. 

Friedman's idea to use classifiers to summarize data before applying classical statistical
analysis downstream can be generalized beyond two-sample testing \eqref{eqn:intro two sample}.
Likelihood-free inference (LFI), also known as simulation-based inference (SBI), has seen a flurry
of interest recently. In LFI, the scientist has a dataset $Z_1,\dots,Z_m
\stackrel{\text{iid}}{\sim} \bb P_{\theta^\star}$ and is given access to a black box simulator
which given a parameter $\theta$ produces a random variable with distribution $\bb P_\theta$. The
goal is to do inference on $\theta^\star$. The key aspect of the problem, lending the name
`likelihood-free`, is that the scientist doesn't know the inner workings of the simulator. In
particular its output is not necessarily differentiable with respect to $\theta$ and the density
of $\bb P_\theta$ cannot be evaluated even up to normalization. This setting arises in numerous
areas of science where highly complex, mechanistic, stochastic simulators are used such as climate
modeling, particle physics, phylogenetics and epidemiology to name a few, and its importance was
realized as early as \cite{diggle1984monte}.  In this paper we study the problem of likelihood-free
hypothesis testing (LFHT) proposed recently in \cite{gerber2022likelihood} as a simplified model of
likelihood-free inference. Compared to two-sample testing, here in addition to the dataset $Z$ of size $m$, we have two
`simulated` samples $X,Y$ of size $n$ each from $\bb P_\sf{X}$ and $\bb P_\sf{Y}$ respectively.
The goal is to test the hypotheses
\begin{equation}\label{eqn:intro lfht}\tag{$\sf{LFHT}$}
    H_0 : Z_i \sim \bb P_\sf{X} \qquad\text{versus}\qquad H_1 : Z_i \sim \bb P_\sf{Y}. 
\end{equation}
It is important that apriori $\bb P_{\sf X}$ and $\bb P_{\sf Y}$ are only known to belong to a certain ambient (usually non-parametric) class. This stands in contrast with the earliest appearances of \eqref{eqn:intro lfht} in \cite{Ziv1988OnCW,gutman1989asymptotically}, where authors studied the rate of decay of the type-I and type-II error probabilities for fixed $\bb P_\sf{X}, \bb P_\sf{Y}$. 

In the context of \eqref{eqn:intro lfht} the idea of Friedman materializes as follows. First,
train a classifier $\cal C:\cal X \to \{0,1\}$ to distinguish between $\bb P_\sf{X}$ and $\bb
P_\sf{Y}$ and second, compare the transformed dataset $\{\cal C(Z_j))\}_{j=1}^m$ to $\{\cal
C(X_i)\}_{i=1}^n$ and $\{\cal C(Y_i)\}_{i=1}^n$. The second step compares iid samples of Bernoulli
random variables (provided $\cal C$ is trained on held out data), thus any reasonable test simply
thresholds the number of $Z_j$ classified as $1$, namely the test is of the form

\begin{equation}\label{eqn:intro scheffe}
    \frac1m \sum_{j=1}^m \cal C(Z_j) \geq \gamma
\end{equation}
for some $\gamma \in [0,1]$. The idea to classify $Z$ as coming from either $\bb P_\sf{X}$ or $\bb P_\sf{Y}$ based on the empirical mass on some separating set $S = \cal C^{-1}(\{1\}) \approx \{\mathrm{d}\bb P_\sf{Y}/\mathrm{d}\bb P_\sf{X}\geq1\}$ has been attributed to Scheff\'e in folklore \cite[Section 6]{devroye2001combinatorial}. To illustrate the genuine importance of these ideas, we draw on the famous Higgs boson discovery. In 2012 \cite{chatrchyan2012observation, adam2015higgs} at the Large Hadron Collider (LHC) a team of physicists announced that they observed the Higgs boson, an elementary particle theorized to exist in 1964. It is regarded as the crowning achievement of the LHC, the most expensive instrument ever built. They achieved this feat via likelihood-free inference, using the ideas of classification accuracy testing/Scheff\'e's test in particular. As part of their analysis pipeline they trained a boosted decision tree classifier on simulated data and thresholded counts of observations falling in the classification region. 

This work was initiated as an attempt to understand the theoretical properties of classifier-accuracy testing, motivated by the clear practical interest in these questions. Our intuition told us that restricting the classifier to have binary output might throw away too much statistical power. In regions with large (small) density ratio, the binary output ought to loose useful information about the (un)certainty of the classifier output. The Neyman-Pearson Lemma phrases this succinctly: the optimal classifier aggregates the log density ratio, while heuristically Scheff\'e's test aggregates indicators that the log density ratio exceeds some threshold. The operational implication of this would be to train probabilistic classifiers $\cal C:\cal X \to \R$ approximating the log density ratio, and to aggregate this $\R$-valued output instead of the binary output. However, our results show that this is not necessary for optimality, at least in the minimax sense.

\subsection{Informal description of the results}\label{sec:intro results}
We study the problems of goodness-of-fit testing, two-sample testing and likelihood-free hypothesis testing in a minimax framework (see Section \ref{sec:fundamental problems} for precise definitions). Namely, given a family of probability distributions $\cal P$, we study the minimum number of observations $n$ (and $m$ for $\LF$) that are required to perform the test with error probability less than $\delta\in(0,1/2)$ in the worst case over the distributions $\bb P_\sf{X}$ and $\bb P_\sf{Y}$. We show for multiple natural classes $\cal P$ that there exist minimax optimal (with some restrictions) classification accuracy tests. 

Let us clarify what we mean by `classification-accuracy' tests for goodness-of-fit testing ($\GoF$) and the problems $\TS$ and $\LF$. Suppose we have a sample $X$ of size $2n$ from the unknown distribution $\P_\sf{X}$. We also have a second sample $Y$  of size $2n$ from $\P_\sf{Y} \in \cal P$ which corresponds to the \emph{known} null distribution in the case of $\GoF$ and is \emph{unknown} in the case of $\TS,\LF$. Finally, for $\LF$ we have an additional sample $Z$ of size $2m$ from $\P_\sf{Z} \in \{\P_\sf{X},\P_\sf{Y}\}$. Write $\cal D_\sf{tr} \eqdef \{X^\sf{tr},Y^\sf{tr},Z^\sf{tr}\}$ for the first halves of each sample and $\cal D_\sf{te} \eqdef \{X^\sf{te},Y^\sf{te}, Z^\sf{te}\}$ for the rest. We train a classifier $\cal C:\cal X \to \{0,1\}$ on the input $\cal D_\sf{tr}$ that aims to assign $1$ to $\bb P_\sf{X}$ and $0$ to $\bb P_\sf{Y}$. Going forward, it will be easier to think of $\cal C$ in terms of the `separating set' $S \eqdef \cal C^{-1}(\{1\})$. Thus, $S$ is a random subset of $\cal X$ whose randomness comes from $\cal D_\sf{tr}$ and potentially an external seed. Given two datasets $\{A_i\}_{i=1}^a,\{B_j\}_{j=1}^b$, we define the classifier-accuracy statistic
\begin{equation}\label{eqn:T_CA def}
    T_S(A,B) \eqdef \frac1a \sum_{i=1}^a \one\{A_i \in S\} - \frac1b\sum_{j=1}^b \one\{B_j \in S\}. 
\end{equation}
The name `classifier-accuracy' is given due to the fact that $T_S (X^\sf{te},Y^\sf{te})+1$ is equal to the sum of the fraction of correctly classified test instances under the two classes. Finally, we say a test is a classifier-accuracy test if its output is obtained by thresholding $|T_S|$ for some classifier $\cal C = \one_S$ on the test data $\cal D_\sf{te}$.

\begin{theorem}[informal]\label{thm:main informal}
    There exist classifier-accuracy tests with minimax (near-)optimal sample complexity for all problems $\GoF,\TS,\LF$ and multiple classes of distributions $\cal P$. 
\end{theorem}

\subsection{Proof sketch}\label{sec:intro proof sketch}

The bulk of the technical difficulty lies in finding a good separating set $S\subseteq \cal X$. But how do we measure the quality of $S$? Define the ``separation'' $\sep(S) \eqdef \P_\sf{X}(S) - \P_\sf{Y}(S)$, and the ``size'' $\tau(S) \eqdef \min\{\P_\sf{X}(S)\P_\sf{X}(S^\sf{c}), \P_\sf{Y}(S)\P_\sf{Y}(S^\sf{c})\}$. The following lemma describes the performance of classifier-accuracy tests \eqref{eqn:T_CA def} in terms of $\sep$ and $\tau$.

\begin{lemma}\label{lemma:test_complexity}
Consider the hypothesis testing problem $H_0: p=q$ versus an arbitrary alternative $H_1$. Suppose that the learner has constructed a separating set $S$ such that $|\sep(S)|=|p(S)-q(S)|\ge \underline{\sep}$ for every $(p,q)\in H_1$, and $\tau(S)=(p(S)(1-p(S))\land(q(S)(1-q(S)))\le \overline{\tau}$ for every $(p,q)\in H_0 \cup H_1$. Then using only the knowledge of $\overline{\tau}$, the classifier-accuracy test \eqref{eqn:T_CA def} with $n$ test samples from both $p$ and $q$ and an appropriate threshold achieves type-I and type-II errors at most $\delta$, provided that
\begin{align*}
    n \ge c \frac{\log(1/\delta)}{\underline{\sep}}\left( 1 + \frac{\overline\tau}{\underline\sep}\right)
\end{align*}
for a large enough universal constant $c>0$. 
\end{lemma}

With Lemma \ref{lemma:test_complexity} in hand it is clear how we need to design $S$. It should satisfy
\begin{equation}\label{eqn:three conditions}
    \big|\sep(S)\big|\text{ is big under $H_1$, and } \tau(S) \text{ is small under both $H_0$ and $H_1$}
\end{equation}
with probability $1-\delta$. The latter condition, namely that $\tau$ is small i.e. $\cal C=\one_S$ is imbalanced, may seem unintuitive as given any two (sufficiently regular) probability distributions there always exists a balanced classifier whose separation is optimal up to constant. 

\begin{proposition}\label{prop:balanced classifier}
Let $\P,\bb Q$ be two distributions on a generic probability space $(\cal X, \mathcal{F})$. Then 
\begin{equation*}
    \TV(\P,\bb Q) \leq 2\sup\{\P(\cal C(X)=0)-\bb Q(\cal C(X)=0):\P(\cal C(X)=0) = \Q(\cal C(X)=1)\}, 
\end{equation*}
where $\cal C: \cal X\to \{0,1\}$ is a possibly randomized classifier. Here the constant $2$ is tight. 
\end{proposition}

Despite Proposition \ref{prop:balanced classifier}, we find that choosing a highly imbalanced classifier $\cal C$ is crucial in obtaining the minimax sample complexity in some classes. This has  interesting implications for practical classifier-accuracy testing. Indeed, classifiers are commonly trained to minimize some proxy of misclassification error; however, the above heuristics show that this is not necessarily optimal, instead one should seek \emph{imbalanced} classifiers with large separation. Another way to phrase it is that when training a classifier for testing one should have the downstream task in mind, namely, maximizing the power of the resulting test, and not classification accuracy.

\subsection{Prior work and contribution} The problem of two-sample ($\sf{TS}$) testing (aka
closeness testing) and the related problem of goodness-of-fit ($\sf{GoF}$) testing (aka identity
testing) has a long history in both statistics and computer science. We only
mention a small subset of the literature, directly relevant to our work. In seminal works Ingster studied
($\sf{GoF}$) for the Gaussian sequence model  \cite{ingster1982minimax,ingster2003nonparametric}
and for smooth densities \cite{ingster1987minimax} in one dimension. Extensions to multiple
dimensions and ($\sf{TS}$) can be found in works such as \cite{li2019optimality,arias2018remember}.
For discrete distributions on a large alphabet the two problems appeared  first in
\cite{goldreich2011testing,batu2000testing}, see also \cite{chan2014optimal,valiant2017automatic} and the survey \cite{canonne2020survey}. Recent work \cite{diakonikolas2018sample,diakonikolas2021optimal} has focused on $\sf{GoF}$ and $\sf{TS}$ with vanishing error probability.

The problem of likelihood-free hypothesis testing appeared first in the works \cite{Ziv1988OnCW,gutman1989asymptotically}, who studied the asymptotic setting. Minimax likelihood-free hypothesis testing ($\sf{LFHT}$) was first studied by the information theory community in \cite{kelly2010universal,kelly2012classification} for a restricted class of discrete distributions on a large alphabet, with a strengthening by \cite{huang2012classification} to vanishing error probability (in some regimes). More recently, the problem was proposed in \cite{gerber2022likelihood} as a simplified model of likelihood-free inference, and authors derived minimax optimal sample complexities for constant error in the settings studied in the present paper. 

The idea of using classifiers for two-sample testing was proposed in \cite{friedman2004multivariate} and has seen a flurry of interest \cite{golland2003permutation,lopez2016revisiting,kim2021classification,hediger2022use}. In likelihood-free inference the output of classifiers can be used as summary statistics for Approximate Bayesian Computation \cite{jiang2017learning,gutmann2018likelihood} or to approximate density ratios \cite{cranmer2020frontier} via the 'likelihood-ratio trick'. A classifier with binary $\{0,1\}$ output was used in the discovery of the Higgs boson \cite{chatrchyan2012observation,adam2015higgs} to determine the detection region. 

Our work is the first to study the non-asymptotic properties of classifier-based tests in any setting and we find that classifier-accuracy tests are minimax optimal for a wide range of problems. As a consequence of our results we resolve the minimax high probability sample complexity of $\LF$ over all classes studied, and also obtain new, tight results on high probability $\GoF$ and $\TS$.

\subsection{Structure}
In Sections \ref{sec:fundamental problems} and \ref{sec:nonparametric classes} we define the statistical problems and distribution classes we study. In Tables \ref{table:prior minimax} and \ref{table:P_D} we present all sample complexity results, and in Section \ref{ssec:minimax complexities} we indicate how to derive them. Sections \ref{sec:discrete sep}, \ref{sec:smooth sep sets} and \ref{sec:gauss sep sets} study the problem of learning good separating sets for discrete and smooth distributions and the Gaussian sequence model respectively. The appendix contains all proofs omitted from the main text, including all lower bounds in Appendix \ref{section:lower bds}.

\section{Results}\label{sec:results}

\subsection{Technical preliminaries}
\subsubsection{Two-sample, goodness-of-fit and likelihood-free hypothesis testing}\label{sec:fundamental problems}
Formally, we define a hypothesis as a set of probability measures. Given two hypotheses $H_0$ and $H_1$ consisting of distributions on some measurable space $\cal X$, we say that a function $\psi:\cal X \to \{0,1\}$ tests
the two hypotheses against each other with error at most $\delta \in (0,1/2)$ if
\begin{equation}\label{eqn:successful test}
\max\limits_{i=0,1} \max\limits_{P \in H_i} \bb P_{S \sim P}(\psi(S) \neq i) \leq \delta.
\end{equation}
% \vspace{-8mm}
Throughout the remainder of this section let $\cal P$ be a class of probability distributions on $\cal X$. Suppose we observe independent samples $X \sim \bb P_{\sf{X}}^{\otimes n}$, $Y \sim \bb P_{\sf{Y}}^{\otimes n}$
and $Z\sim\bb P_{\sf{Z}}^{\otimes m}$ whose distributions $\bb P_{\sf{X}},\bb P_{\sf{Y}}, \bb P_{\sf{Z}} \in \cal P$ are
 \it{unknown} to us. We now define the problems at the center of our work. 

\begin{definition}
Given a known $\bb P_0 \in \cal P$, \textbf{goodness-of-fit testing} is the comparison of
\begin{equation}\label{eqn:GoF definition}\tag{$\sf{GoF}$}
H_0: \bb P_\sf{X}=\bb P_0 \qquad\text{versus}\qquad H_1: \TV(\bb P_\sf{X}, \bb P_0) \geq \epsilon
\end{equation}
based on the sample $X$. Write $n_\GoF(\epsilon, \delta, \cal P)$ for the smallest number such that for all $n \geq n_\sf{TS}$ there exists
a function $\psi : \cal X^n \to \{0,1\}$ which given $X$ as input tests between
$H_0$ and $H_1$ with error probability at most $\delta$, for arbitrary $\bb P_\sf{X},\bb P_0\in\cal P$. 
\end{definition}

\begin{definition}
\textbf{Two-sample testing} is the comparison of
\begin{equation}\label{eqn:TS definition}\tag{$\sf{TS}$}
H_0: \bb P_\sf{X}=\bb P_\sf{Y} \qquad\text{versus}\qquad H_1: \TV(\bb P_\sf{X}, \bb P_\sf{Y}) \geq \epsilon
\end{equation}
based on the samples $X,Y$. Write $n_\sf{TS}(\epsilon, \delta, \cal P)$ for the smallest number such that for all $n \geq n_\sf{TS}$ there exists
a function $\psi : \cal X^n \times \cal X^n \to \{0,1\}$ which given $X,Y$ as input tests between
$H_0$ and $H_1$ with error probability at most $\delta$, for arbitrary $\bb P_\sf{X},\bb P_\sf{Y}\in\cal P$.
\end{definition}

\begin{definition}
\textbf{Likelihood-free hypothesis testing} is the comparison of
\begin{equation}\label{eqn:LFHT definition}\tag{$\sf{LF}$}
H_0: \bb P_\sf{Z}=\bb P_\sf{X} \qquad\text{versus}\qquad H_1: \bb P_\sf{Z}=\bb P_\sf{Y}
\end{equation}
based on the samples $X,Y,Z$. Write $\cal R_\sf{LF}(\epsilon, \delta, \cal P) \subseteq \bb R^2$ for the maximal set such that for all $(n,m)\in\bb N^2$ with $n\geq x, m\geq y$ for some $(x,y)\in\cal R_\sf{LF}$, there exists a function $\psi : \cal X^n \times \cal X^n\times\cal X^m \to \{0,1\}$ which given $X,Y,Z$ as input,
successfully tests $H_0$ against $H_1$ with error probability at most $\delta$, provided $\TV(\bb P_\sf{X}, \bb P_\sf{Y}) \geq \epsilon$ and $\bb P_\sf{X}, \bb P_\sf{Y} \in \cal P$.
\end{definition}

\subsubsection{Classes of distributions}\label{sec:nonparametric classes}
We consider the following nonparametric families of distributions. 

\noindent \textbf{Smooth density.} Let $\cal C(\beta, d, C)$ denote the set of functions $f:[0,1]^d \to \bb R$ that are $\lceil \beta-1\rceil$-times differentiable and satisfy
\begin{align*}
    \|f\|_{\cal C_\beta} \eqdef \max\left(\max\limits_{0\leq |\alpha| \leq \lceil \beta-1\rceil} \|f^{(\alpha)}\|_\infty, \sup\limits_{x\neq y \in [0,1]^d, |\alpha|=\lceil \beta-1\rceil} \frac{|f^{(\alpha)}(x)-f^{(\alpha)}(y)|}{\|x-y\|^{\beta-\lceil \beta-1\rceil}_2}\right) \leq C, 
\end{align*}
where $\lceil \beta-1\rceil$ denotes the largest integer strictly smaller than $\beta$ and $|\alpha|=\sum_{i=1}^d\alpha_i$ for the multiindex $\alpha \in \N^d$. We write $\cal P_\sf{H}(\beta, d, C_\sf{H})$ for the class of distributions with Lebesgue-densities in $\cal C(\beta, d, C_\sf{H})$.

\noindent \textbf{Distributions on a finite alphabet.} For $k \in \bb N$, let
\begin{align*}
  \cal P_\sf{D}(k) &\eqdef \{\textrm{all distributions on the finite alphabet }[k]\}, \\
    \cal P_\sf{Db}(k, C_\sf{Db}) &\eqdef \{p \in \cal P_\sf{D}(k): \|p\|_\infty \leq C_\sf{Db}/k\},
\end{align*}
where $C_\sf{Db} > 1$ is a constant. In other words, $\cal P_\sf{Db}$ are those discrete
distributions that are bounded by a constant multiple of the uniform distribution. 

\noindent \textbf{Gaussian sequence model on the Sobolev ellipsoid.}
Define the Sobolev ellipsoid $\cal E(s, C)$ of smoothness $s > 0$ and size $C>0$ as
$\{\theta\in\bb R^{\bb N}: \sum_{j=1}^\infty j^{2s} \theta_j^2 \leq C\}.$
For $\theta \in \R^\infty$ let $\mu_\theta = \otimes_{i=1}^\infty \cal N(\theta_i, 1)$, and define our second class as
\begin{equation*}
    \cal P_\sf{G}(s, C_\sf{G}) \eqdef \left\{\mu_\theta\,: \theta \in \cal E(s, C_\sf{G})\right\}.
\end{equation*}
To briefly motivate the study of $\cal P_\sf{G}$, consider the classical Gaussian white noise model. Here we have iid observations of the stochastic process
\begin{equation*}
    \D Y_t = f(t)\D t + \D W_t,\,\,t\in[0,1],
\end{equation*}
where $(W_t)_{t\geq0}$ denotes Brownian motion and $f\in L^2[0,1]$ is unknown. Suppose now that $\{\phi_i\}_{i\geq1}$ forms an orthonormal basis for $L^2[0,1]$ and given an observation $Y$ define the values
\begin{align*}
    y_i &\eqdef \left\langle Y, \phi_i\right\rangle = \int_0^1 f(t) \phi_i(t) \D t +  \int_0^1 \phi_i(t)\D W_t \eqdef \theta_i + \epsilon_i. 
\end{align*}
Notice that $\epsilon_i \sim \cal N(0,1)$ and that $\E[\epsilon_i\epsilon_j] = \one_{i=j}$. In other words, the sequence $\{y_i\}_{i\geq1}$ is an observation from the distribution $\mu_\theta$. Consider the particular case of $\phi_1\equiv1$ and $\phi_{2k}=\sqrt 2\cos(2\pi kx),\phi_{2k+1}=\sqrt2\sin(2\pi kx)$ for $k\geq1$ and assume that $f$ satisfies periodic boundary conditions. Then $\theta$ denotes the Fourier coefficients of $f$ and the condition that $\sum_{j=1}^\infty j^{2s}\theta_j^2 \leq C$ is equivalent to an upper bound on the order $(s,2)$-Sobolev norm of $f$, see e.g. Proposition 1.14 of \cite{tsybakov}. In other words, by studying the class $\cal P_\sf{G}$ we can deduce results for signal detection in Gaussian white noise, where the signal has bounded Sobolev norm.

\begin{table}[!h]\caption[]{Minimax sample complexity of testing (up to constant factors) over $\cal P_\sf{H}, \cal P_\sf{G}, \cal P_\sf{Db}$.}
\label{table:prior minimax}
\def\arraystretch{2}

\fontsize{12pt}{12pt}\selectfont
\begin{center}
\begin{tabular}{ c|c |c | c }
 & \cellcolor{gray!10}$n_\GoF$&\cellcolor{gray!10}$n_\sf{TS}$ &\cellcolor{gray!10} $\cal R_\sf{LF}$ \\ [0.5ex]
 \hline

 \cellcolor{gray!10}
 $\cal P_\sf{Db}(k)$ & \cellcolor{green!0}$\frac{\sqrt{k\log(1/\delta)}}{\epsilon^2} + \frac{\log(1/\delta)}{\epsilon^2}$& \cellcolor{green!0}$n_\GoF$ & \cellcolor{green!0}$m\geq\frac{\log(1/\delta)}{\epsilon^{2}}\text{ and } n\geq n_\sf{GoF}\text{ and } n m \geq n_\GoF^2$ \\
 \hline

\cellcolor{gray!10}
$\cal P_\sf{H}(\beta, d)$ & \cellcolor{green!0}$\frac{\sqrt{\log(1/\delta)}}{\epsilon^{(2\beta+d/2)/\beta}} + \frac{\log(1/\delta)}{\epsilon^2}$ & \cellcolor{green!0}$n_\GoF$ & \cellcolor{green!0}$m\geq\frac{\log(1/\delta)}{\epsilon^{2}}\text{ and } n\geq n_\sf{GoF}\text{ and } n m \geq n_\GoF^2$ \\
 \hline

 \cellcolor{gray!10}
 $\cal P_\sf{G}(s)$ & \cellcolor{green!0}$\frac{\sqrt{\log(1/\delta)}}{\epsilon^{(2s+1/2)/s}} + \frac{\log(1/\delta)}{\epsilon^2}$&\cellcolor{green!0} $n_\GoF$ &\cellcolor{green!0} $m\geq\frac{\log(1/\delta)}{\epsilon^{2}}\text{ and } n\geq n_\sf{GoF}\text{ and } n m \geq n_\GoF^2$ \\

\end{tabular}
\end{center}
\end{table}

\begin{table}
% \vspace{-8mm}
\begin{center}
\caption{Minimax sample complexity of testing (up to constant factors) over $\cal P_\sf{D}$.}
\label{table:P_D}
\def\arraystretch{2}
\fontsize{9pt}{9pt}\selectfont
\begin{NiceTabular}{ccccc}[hvlines-except-borders,rules/width=.2pt]

    \RowStyle{\bfseries}
     & \cellcolor{gray!10}{\large $n_\GoF(\cal P_\sf{D})$} & \cellcolor{gray!10}{\large $n_\TS(\cal P_\sf{D})$} & \cellcolor{gray!10}\Block{1-2}{{\large $\cal R_\sf{LF}(\cal P_\sf{D})$}} &\cellcolor{gray!10}\\
     % \hline
    
    \cellcolor{gray!10}\Block{3-1}{{$k\geq\frac{\log\left(\frac{1}{\delta}\right)}{\epsilon^4}$}} & \Block{2-1}{$(\sf{OPT})\,n_\GoF(\cal P_\sf{Db})$} & \cellcolor{green!0}\Block{3-1}{\cellcolor{green!0}$\left(\frac{k^{2}\log\left(\frac{1}{\delta}\right)}{\epsilon^{4}}\right)^{\frac13}$} & \cellcolor{gray!10}$n\geq m$ & \cellcolor{green!0}$m\geq \frac{\log(1/\delta)}{\epsilon^2}\text{ and }m\min\{n^2/k,n\} \geq n_\GoF^2$\\
    
    \cellcolor{gray!10}& & \cellcolor{green!0}& \cellcolor{gray!10}\Block{2-1}{$m > n$} & $(\sf{OPT})\, mn^2 \geq kn^2_\GoF \text{ and }n \geq n_\sf{GoF}$ \\
    
    \cellcolor{gray!10}& $(\sf{CAT})\,n_\GoF\left(\frac{\epsilon}{\textcolor{red}{\log(k)}}, \frac{\delta}{\textcolor{red}{k}}, \cal P_\sf{Db}\right)$ & \cellcolor{green!0}& \cellcolor{gray!10} & \cellcolor{red!0} ($\sf{CAT}$) \begin{tabular}{c} $\frac{mn^2}{\textcolor{red}{\log\left(\frac k\delta \right)}} \geq kn^2_\sf{GoF}\left(\frac{\epsilon}{\textcolor{red}{\log(k)}}, \frac{\delta}{\textcolor{red}{k}}\right)$ \\\text{and } $n \geq n_\sf{GoF}(\frac{\epsilon}{\textcolor{red}{\log(k)}}, \frac{\delta}{\textcolor{red}{k}})$ \end{tabular} \\
    
    % \hline
    \cellcolor{gray!10}{$k<\frac{\log\left(\frac{1}{\delta}\right)}{\epsilon^4}$ }&\cellcolor{green!0} $n_\GoF(\cal P_\sf{Db})$ & \cellcolor{green!0}$n_\GoF(\cal P_\sf{Db})$ & \cellcolor{green!0}\Block{1-2}{$m\geq\frac{\log(1/\delta)}{\epsilon^2}\text{ and } n\geq n_\GoF \text{ and } n m \geq n_\GoF^2$} & \cellcolor{green!0}
\end{NiceTabular}
\end{center}
\end{table}

\subsection{Minimax sample complexity of classifier-accuracy tests}\label{ssec:minimax complexities}
In Tables\footnote{We thank Gunjan Kumar for pointing out an error in the top right cell of Table \ref{table:P_D} in a previous version.} \ref{table:prior minimax} and \ref{table:P_D} we present our and prior results on the minimax sample complexity of $\GoF,\TS$ and $\LF$; here 
\begin{itemize}
    \item unmarked entries denote minimax optimal results achievable by a classifier-accuracy test; 
    \item entries marked with $(\sf{OPT})$ denote minimax optimal results that are not known to be achievable by any classifier-accuracy test; 
    \item entries marked with $(\sf{CAT})$ denote the best known result using a classifier-accuracy test. 
\end{itemize}

In the constant error regime ($\delta=\Theta(1)$) the results of Tables \ref{table:prior minimax} and \ref{table:P_D} are well known; for instance, the sample complexities of $\GoF$, $\TS$, and $\LF$ under $\cal P_{\sf D}$ were characterized in \cite{paninski2008coincidence,bhattacharya2015testing,gerber2022likelihood}, respectively\footnote{\cite{gerber2022likelihood} only resolved the minimax sample complexity of $\LF$ for $\cal P_\sf{D}$ up to $\log(k)$-factors in some regimes. However, by combining the classifier accuracy tests of this paper for $m \leq n$ and the reduction to two-sample testing with unequal sample size \cite{bhattacharya2015testing, diakonikolas2021optimal} for $m > n$ these gaps are filled.}. Less is known under the high-probability regime ($\delta=o(1)$): for $\cal P_{\sf D}$, $n_\GoF$ was characterized in \cite{huang2013generalized, diakonikolas2018sample} for uniformity testing, with the general case following from the flattening reduction \cite{diakonikolas2016new}; $n_\TS$ was characterized in \cite{diakonikolas2021optimal}. For $\cal R_\sf{LF}$, the $k>n$ case for $\cal P_\sf{Db}$ is resolved by \cite{huang2012classification}, and the achievability direction of the case $m > n$ of $\cal R_\sf{LF}$ for $\cal P_\sf{D}$ can be deduced from \cite{diakonikolas2021optimal} via the natural reduction between $\TS$ and $\LF$ (see \cite{gerber2022likelihood}). The remaining upper bounds are achievable by the classifier-accuracy tests below, and the proofs of all lower bounds are deferred to Appendix \ref{section:lower bds}. 

As for the efficacy of classifier-accuracy tests, the upper bounds in Tables \ref{table:prior minimax} and \ref{table:P_D} follow from the combination of Lemma \ref{lemma:test_complexity} and the following results: 
\begin{itemize}
    \item $\pmb{\cal P_\sf{Db}}:$ see Corollary \ref{cor:sep+tau for P_Db}; 
    \item $\pmb{\cal P_\sf{H}}:$ see Section \ref{sec:smooth sep sets} and Corollary \ref{cor:sep+tau for P_Db}; 
    \item $\pmb{\cal P_\sf{G}}:$ see Proposition \ref{prop:finding gauss sepset}; 
    \item $\pmb{\cal P_\sf{D}}:$ for $\GoF$, see Proposition \ref{prop:S_1/2 E sep} if $k < \log(1/\delta)/\epsilon^4$, and Proposition \ref{prop:sep + tau guarantees P_D} otherwise; for $\TS$, see Proposition \ref{prop:S_1/2 E sep}; for $\LF$, see Proposition \ref{prop:S_1/2 E sep} if $n\ge k \wedge m$, and Section \ref{sec:best of log(k)} and Proposition \ref{prop:sep + tau guarantees P_D} otherwise. 
\end{itemize}

\section{Learning separating sets}\label{sec:learning sep sets}
In this section, we construct the separating sets $S$ used in the classifier-accuracy test \eqref{eqn:T_CA def}. Section \ref{sec:discrete sep} is devoted to discrete distribution models $\cal P_{\sf{Db}}$ and $\cal P_{\sf D}$, where we need a delicate tradeoff between the expected separation and the size of $S$. A similar construction in the Gaussian sequence model $\cal P_{\sf G}$ is presented in Section \ref{sec:gauss sep sets}. 

\subsection{The discrete case}\label{sec:discrete sep}
Given two iid samples $X,Y$ of sizes $N_X, N_Y \stackrel{iid}{\sim} \poi(n)$ from unknown discrete distributions $p=(p_1,\dots,p_k),q=(q_1,\dots,q_k)$ over a finite alphabet $[k]=\{1,2,\dots,k\}$, can we learn a set $\hat S \subseteq [k]$ using $X,Y$ that separates $p$ from $q$? To measure the quality of a given separating set $A\subseteq [k]$, we define two quantities $\sep(A) \eqdef p(A)-q(A)$ and $\tau(A) \eqdef \min\{p(A)p(A^\sf{c}), q(A)q(A^\sf{c})\}$. Intuitively, the first quantity $\sep(A)$ measures the separation of $A$, and the second quantity $\tau(A)$ measures the size of $A$. Recall that by Lemma \ref{lemma:test_complexity}, in order to perform the classifier-accuracy test \eqref{eqn:T_CA def}, we aim to find a separating set $\hat S$ such that
\begin{equation}\label{eqn:desired sep set}
    |\sep(\hat S)| \text{ is large and } \tau(\hat S) \text{ is small.}
\end{equation}
The rest of this section is devoted to the construction of $\hat S$ satisfying \eqref{eqn:desired sep set}, and we will present our results on learning separating sets in order of increasing complexity. 

\emph{Notation:} for a random variable $X$ we write $\sigma^2(X)$ for the optimal sub-Gaussian variance proxy of $X$. In other words, $\sigma^2(X)$ is the smallest value such that $\E\exp(\lambda (X-\E X)) \leq \exp(\lambda^2\sigma^2(X)/2)$ holds for all $\lambda \in \R$. 

\subsubsection{A natural separating set}\label{sec:trivial classifier}
Let $\{X_i,Y_i\}_{i \in [k]}$ be the empirical frequencies of each bin $i \in [k]$ in our samples $X,Y$, i.e. $nX_i\sim \poi(np_i)$ and $nY_i\sim \poi(nq_i)$. A natural separating set is the following: 
\begin{align*}
    \hat S_{1/2} &\eqdef \{i: X_i > Y_i \text{ or } X_i=Y_i \text{ and } C_i=1\}, 
\end{align*}
where $C_1,C_2\dots C_k$ are iid $\ber(1/2)$ random variables. We use the subscript ``$1/2$'' to illustrate our tie-breaking rule: when $X_i = Y_i$, the symbol $i$ is added to the set with probability $1/2$. 

Our first result concerns the separating power of the above set. 
\begin{proposition}\label{prop:S_1/2 E sep}
Suppose $p,q \in \cal P_\sf{D}(k)$ with $\TV(p,q)\ge \epsilon$. There exists a universal constant $c>0$ such that 
\begin{align*}
    \P\left(\sep(\hat S_{1/2}) \geq c\epsilon^2\left(\frac nk \land \sqrt{\frac nk} \land \frac1\epsilon\right)\right) \geq 1-\delta, 
\end{align*} 
provided $n \geq \frac1c n_\sf{TS}(\epsilon, \delta, \cal P_\sf{D}(k))$. 
\end{proposition}

Together with the trivial upper bound $\tau(\hat S_{1/2})\le 1/4$, Proposition \ref{prop:S_1/2 E sep} and Lemma \ref{lemma:test_complexity} imply that using $\hat S_{1/2}$ achieves the minimax sample complexity for the following problems: 
\begin{itemize}
    \item $\GoF$ in $\cal P_{\sf{Db}}$ and $\cal P_{\sf D}$ as long as $k= \cal O(\log(1/\delta)/\epsilon^4)$; 
    \item $\TS$ in $\cal P_{\sf{Db}}$ as long as $k= \cal O(\log(1/\delta)/\epsilon^4)$, and in $\cal P_{\sf D}$ for all $(k,\epsilon,\delta)$; 
    \item $\LF$ in $\cal P_{\sf{Db}}$ as long as $k=\cal O(\log(1/\delta)/\epsilon^4)$, and in $\cal P_{\sf D}$ as long as $n\ge m$. 
\end{itemize}

However, in the remaining regimes the above test could be strictly sub-optimal. This failure comes down to two issues. First, Proposition \ref{prop:S_1/2 E sep} requires $n\gtrsim n_{\TS}(\epsilon,\delta,\cal P_{\sf D}(k))$ in order to find a good separating set, which can be sub-optimal when the optimal sample complexity for the original testing problem is only $n\gtrsim n_{\GoF}(\epsilon,\delta,\cal P_{\sf D}(k))$. Second, the quantity $\tau(\hat S_{1/2})$ is $\Omega(1)$ in the general case because the tie-breaking rule adds too many symbols to the set. These issues will be addressed separately in the next two sections.

\subsubsection{The ``better of two'' separating sets}
This section aims to find a separating set $\hat S$ with essentially the same separation as $\hat S_{1/2}$ in Proposition \ref{prop:S_1/2 E sep}, but with a smaller $\tau(\hat S)$. The central idea is to use a different tie-breaking rule from $\hat S_{1/2}$. Given a subset $D \subseteq [k]$, we define the imbalanced separating sets
\begin{align*}
    \hat S_>(D) &= \{i\in D:X_i>Y_i\}, \\
    \hat S_<(D) &= \{i \in D:X_i < Y_i\}. 
\end{align*}
In other words, in both $\hat S_>$ and $\hat S_<$, we do not include the symbols with $X_i = Y_i$ in the separating set. Consequently, $|\hat S_>(D)| \lor |\hat S_<(D)|$ is upper bounded by the sample size; if in addition $q_i$ is bounded from above uniformly over $i\in D$, this will yield good control of $\tau$ for both separating sets $\hat S_>(D)$ and $\hat S_<(D)$. In particular, $\tau(\hat S_>(D)) \lor \tau(\hat S_<(D)) =\cal O( 1\land (n\max_{i\in D}q_i))$. 

Next we aim to show that the above sets achieve good separation. However, there is a subtlety here: removing the ties from $\hat S_{1/2}$ may no longer guarantee the desired separation, as illustrated in the following proposition. 

\begin{proposition}\label{prop:neg sep}
    Consider the distributions $p,q$ on $[3k]$ with $p_i = \one\{i\leq k\}/(2k) + \one\{i > k\}/(4k)$ and $q_i = \one\{i\leq k\}/k$. Then, for $n \leq 0.6k$, 
    \begin{equation*}
        \E \operatorname{sep}(\hat S_>([3k])) < 0. 
    \end{equation*} 
\end{proposition}

Proposition \ref{prop:neg sep} shows that sticking to only one set $\hat S_>$ or $\hat S_<$ fails to give the same separation guarantees as Proposition \ref{prop:S_1/2 E sep}. A priori it may seem that
$\hat S_>$ is designed to capture elements of the support where $p$ is greater than $q$, but it
fails to do so spectacularly. An intuitive explanation of this phenomenon is as follows. Since the
probability of each bin is small ($\lesssim1/k$) under both $p$ and $q$, in the small $n$
regime\footnote{Technically, to satisfy the stated conditions we would require $n\lesssim
\sqrt{k}$, but the described event captures dominant
effects even for larger $\sqrt{k} \ll n\ll k$.}
can expect that $(a)$ each bin appears either once or not at all and $(b)$ there is no overlap
between the observed bins in sample $X$ and $Y$. In this heuristic picture, the set $\hat S_>$ is
simply the set of observed bins in the $X$-sample. Each $X$-sample falling in
the first $k$ bins contributes $-{1\over 2k}$ to  the separation, while each $X$-sample in the last
$2k$ bins contributes only $+{1\over 4k}$ to the separation.  Since $p$ puts mass $1/2$ on both the first $k$
and last $2k$ bins, there is an equal number of $n/2$ observations in each part and the overall separation is
$\asymp {-{n\over 8k}}$. Similar results can be proved for $\hat S_<$ with $p,q$ as above but swapped,
and also for modified $p,q$ separated by smaller $\epsilon$ in $\TV$ for any $\epsilon \in (0,1)$. 

Motivated by the above discussion, in the sequel we consider the sets $\hat S_>, \hat S_<$ jointly. Specifically, the next proposition shows that \emph{at least one of} the sets $\hat S_>$ and $\hat S_<$ have a good separation. 

\begin{proposition}\label{prop:S_> + S_geq E lower}
    There exists a universal constant $c>0$ such that for any $D \subseteq [k]$ and probability mass functions $p,q$, it holds that
    \begin{align*}
        \E\left[\sep(\hat S_>(D))-\sep(\hat S_<(D))\right] &\geq c\sum_{i \in D} \frac{n(p_i-q_i)^2}{\sqrt{n(p_i\wedge q_i)+1}} \land |p_i-q_i|, \\
        \sigma^2(\sep(\hat S_>(D))) + \sigma^2(\sep(\hat S_<(D))) &\leq \frac 1c \sum_{i \in D} \frac{p_i+q_i}{n} \land |p_i-q_i|^2. 
    \end{align*}
\end{proposition}
Based on Proposition \ref{prop:S_> + S_geq E lower}, our final separating set is chosen from these two options, based on evaluation on held out data. As for the choice of $D$, in this section we choose $D = [k]$. The following corollary summarizes the performance of this choice under $\cal P_\sf{Db}$.

\begin{corollary}\label{cor:sep+tau for P_Db}
    Suppose $p,q \in \cal P_\sf{Db}(k,\cal O(1))$ with $\TV(p,q)\ge \epsilon$. There exists a universal constant $c>0$ such that using the samples $X,Y$ we can find a set $\hat S \subseteq [k]$ which, with probability $1-\delta$, satisfies
    \begin{equation}
        \left|\sep(\hat S)\right| \geq c\epsilon^2\left(\frac1\epsilon\land\sqrt{\frac nk}\land\frac nk \right) \qquad\text{and}\qquad \tau(\hat S) \leq \frac1c\left(1\land \frac nk\right),
    \end{equation}
    provided $n\geq \frac1cn_\GoF(\epsilon, \delta, \cal P_\sf{Db}(k,\cal O(1)))$. 
\end{corollary}

By Corollary \ref{cor:sep+tau for P_Db} and Lemma \ref{lemma:test_complexity}, using the above set $\hat S$ achieves the minimax sample complexity for all problems $\GoF$, $\TS$, and $\LF$ and all parameters $(k,\epsilon,\delta)$ under $\cal P_{\sf{Db}}$. However, under $\cal P_{\sf D}$, the performance of $\hat S$ is no better than that of $\hat S_{1/2}$. This is because a good control of $\tau(\hat S_>([k]))$ requires a bounded probability mass function; in other words, choosing $D=[k]$ is not optimal for finding the best separating set under $\cal P_{\sf D}$. In the next section, we address this issue by choosing $D$ to be one of $\cal O(\log k)$ subsets of $[k]$. 

\subsubsection{The ``best of $\cal O(\log k)$" separating sets}\label{sec:best of log(k)}
This section is devoted to the two missing regimes $m \geq n$ for $\LF$ over $\cal P_\sf{D}$ and $k \gtrsim \log(1/\delta)/\epsilon^4$ for $\GoF$ over $\cal P_\sf{D}$ (cf. discussion after Proposition \ref{prop:balanced classifier} and Corollary \ref{cor:sep+tau for P_Db}). For the former, recall that the classifier-accuracy test based on $\hat S_{1/2}$ achieves the sample complexity
\begin{align}\label{eqn:S_1/2 performance}
    n \gtrsim n_{\GoF}(\epsilon,\delta,\cal P_{\sf D}) + \frac{k\sqrt{\log(1/\delta)}}{\sqrt{n}\epsilon^2}.
\end{align}
If $n \gtrsim k$ then \eqref{eqn:S_1/2 performance} is the same as $n\gtrsim n_\GoF$; if $m/\log(1/\delta) \lesssim n$ then \eqref{eqn:S_1/2 performance} is implied by $n\gtrsim n_\GoF + \frac{k\log(1/\delta)}{\sqrt m\epsilon^2}$, which is optimal within an $O(\log^{1/2}(1/\delta))$ factor (cf. Table \ref{table:P_D}). In our application to $\GoF$ we take $m=\infty$, and the missing regime $k \gtrsim \log(1/\delta)/\epsilon^4$ corresponds precisely to $n_\GoF \lesssim k$. 
Summarizing, in the remainder of this section we may assume that $k \land (m/\log(1/\delta)) \gtrsim n$.

Let $t = k\wedge (c_0m/\log(1/\delta))$, where $c_0>0$ is a small absolute constant. By the previous paragraph, we assume without loss of generality that $t > n$. For $\ell = \lceil \log_2(t/n)\rceil \ge 1$, define the following $\ell+2$ subsets of $[k]$:  
\begin{align*}
    D_0 &= \left\{i: \hat{q}_i^0 \leq \frac1t\right\}, \quad
    D_j = \left\{i: \hat{q}_i^0 \in \Big(\frac{2^{j-1}}{t}, \frac{2^j}{t}\Big]\right\} \text{ for }j\in[\ell], \quad
    D_{\ell+1} = \left\{i: \hat{q}_i^0 > \frac{2^\ell}{t}\right\}. 
\end{align*}
Here $\hat{q}_i^0$ denotes the empirical pmf of $m/2$ held out samples drawn from $q$ (for $\GoF$, one can understand $\hat{q}_i^0 = q_i$ for the distribution $q$ is known). The motivation behind the above choices is the ``localization'' of each $\hat{q}_i^0$, as shown in the following lemma. 

\begin{lemma}\label{lemma:localization}
For a small enough universal constant $c_0>0$, with probability at least $1-k\delta$ it holds that for each $i\in [k]$: 
\begin{enumerate}
    \item if $\hat{q}_i^0\in D_0$, then $q_i < 2/t$; 
    \item if $\hat{q}_i^0\in D_j$ for some $j\in [\ell]$, then $q_i \in (2^{j-2}/t, 2^{j+1}/t]$;
    \item if $\hat{q}_i^0\in D_{\ell+1}$, then $q_i > 2^{\ell-1}/t$. 
\end{enumerate}
\end{lemma}

Lemma \ref{lemma:localization} ensures that with high probability, the distribution $q$ restricted to each set $D_j$ is near-uniform. This is similar in spirit to the idea of flattening used in distribution testing \cite{diakonikolas2016new}. The proof of Lemma \ref{lemma:localization} directly follows from the Poisson concentration in Lemma \ref{lem:poisson tail} and is thus omitted. 

Our main result of this section is the next proposition, which shows that there exist some $j\in \{0,1,\cdots,\ell+1\}$ and $\hat S\subseteq D_j$ such that $\hat S$ is a near-optimal separating set within logarithmic factors.

\begin{proposition}\label{prop:sep + tau guarantees P_D}
Suppose $p, q\in \cal P_{\sf D}(k)$ with $\TV(p,q)\ge \epsilon$, and $X, Y$ are $n$ iid samples drawn from $p, q$ respectively. There exists a universal constant $c>0$ such that using the samples $X, Y$, we can find some $j\in \{0,1,\cdots,\ell+1\}$ and a set $\hat S\subseteq D_j$ which, with probability $1-\cal O(k\delta)$, satisfies 
\begin{align*}
    \left|\sep(\hat S)\right| \ge c\left(\frac{\epsilon}{\ell}\right)^2\begin{cases}\begin{rcases}
       n/k &\text{if } j = 0 \\
       n/\sqrt{kt/2^j} & \text{if } j \in [\ell+1]
    \end{rcases}\end{cases} \quad \text{and} \quad \tau(\hat S)\le \frac{n2^j}{ct}
\end{align*}
provided that $$n\sqrt{1\land\frac{m}{\log(1/\delta)k}} \ge \frac{1}{c} n_\GoF(\epsilon/\ell,\delta,\cal P_{\sf D}).$$
\end{proposition}

By Proposition \ref{prop:sep + tau guarantees P_D} and Lemma \ref{lemma:test_complexity}, using the above set $\hat S$ leads to the following sample complexity guarantee for the problems $\GoF$ and $\LF$: 
\begin{itemize}
    \item for $\GoF$ under $\cal P_{\sf D}$, it succeeds with $n=\Theta(n_\GoF(\epsilon/\ell, \delta /k, \cal P_\sf{D}))$ observations, which is within a multiplicative $\cal O(\log^{\Theta(1)}(k))$ factor of the minimax optimal sample complexity in the missing $k\geq\log(1/\delta)/\epsilon^4$ regime; 
    \item for $\LF$ under $\cal P_{\sf D}$ and $m\ge n$, it succeeds with $n=\Theta(n_\GoF(\epsilon/\ell,\delta/k,\cal P_\sf{D})\sqrt{k\log(k/\delta)/m})$ observations, which is within a multiplicative $\cal O(\log^{\Theta(1)}(k)\log(k/\delta))$ factor of the minimax optimal sample complexity in the missing $n \leq m\land k$. 
\end{itemize}
Therefore, classifier-accuracy tests always lead to near-optimal sample complexities for all $\GoF, \TS$, and $\LF$ problems under both $\cal P_{\sf{Db}}$ and $\cal P_{\sf D}$, within polylogarithmic factors in $(k, 1/\delta)$. We leave the removal of extra logarithmic factors for classifier-accuracy tests as an open problem.

\subsection{The smooth density case}\label{sec:smooth sep sets}
We briefly explain how Corollary \ref{cor:sep+tau for P_Db} can be used to learn separating sets between distributions in the class $\cal P_\sf{H}$ of $\beta$-H\"older smooth distributions on $[0,1]^d$. The reduction relies on an approximation result due to Ingster \cite{ingster1987minimax,ingster2003nonparametric}, see also \cite[Lemma 7.2]{arias2018remember}. Let $P_r$ be the $L^2$-projection onto piecewise constant functions on the regular grid on $[0,1]^d$ with $r^d$ cells. 

\begin{lemma}\label{lem:Ingster approx}
    There exist constants $c_1,c_2$ independent of $r$ such that for any $f \in \cal P_\sf{H}(\beta, d, C_\sf{H})$, 
    \begin{equation*}
        \|P_rf\|_2\geq c_1\|f\|_2-c_2r^{-\beta}. 
    \end{equation*}
\end{lemma}

For simplicity write $f,g$ for the Lebesgue densities of $\bb P_\sf{X},\bb P_\sf{Y} \in \cal P_\sf{H}$. Suppose $\TV(\bb P_\sf{X}, \bb P_\sf{Y}) = \frac12\|f-g\|_1 \geq \epsilon$. By Jensen's inequality and Lemma \ref{lem:Ingster approx}, $\epsilon \lesssim \|P_r(f-g)\|_2$ for $r \asymp \epsilon^{-1/\beta}$. The key observation is that $P_rf$ is essentially the probability mass function of the distribution $\bb P_\sf{X}$ when binned on the regular grid with $r^d$ cells. We can now directly apply the results for $\cal P_\sf{Db}$ (Corollary \ref{cor:sep+tau for P_Db}) with alphabet size $k \asymp \epsilon^{-d/\beta}$, which combined with Lemma \ref{lemma:test_complexity} leads to the sample complexity guarantees in Table \ref{table:prior minimax} for the smooth density class $\cal P_\sf{H}$ in all three problems $\GoF,\TS$ and $\LF$.

\subsection{The Gaussian case}\label{sec:gauss sep sets}

Suppose we have two samples $X,Y$ of size $n$ from $\otimes_{j=1}^\infty \cal N(\theta^X_j, 1) =: \mu_{\theta^X}$ and $\mu_{\theta^Y}$ respectively, where $\theta^X,\theta^Y$ have Sobolev norm $\|\theta\|_s^2 \eqdef \sum_j \theta_j^2 j^{2s}$ bounded by a constant. In addition, $\TV(\mu_{\theta^X}, \mu_{\theta^Y}) \geq \epsilon > 0$. We use $\hat\theta^X$ and $\hat\theta^Y$ to denote the empirical mean vector from samples $X$ and $Y$, respectively.

The separating set is constructed as follows: 
\begin{equation*}
    \hat S = \{Z \in \R^\N : T(Z) \geq 0\}, 
\end{equation*}
where $T(Z) =  2\sum_{j=1}^J (\hat\theta^X_j-\hat\theta^Y_j)(Z_j - (\hat\theta^X_j+\hat\theta^Y_j)/2)$ for some $J \in \mathbb{N}$ to be specified. This is simply a truncated version of the likelihood-ratio test between $\mu_{\hat\theta^X}$ and $\mu_{\hat\theta^Y}$, where we set all but the first $J$ coordinates of  $\hat\theta^X$ and $\hat\theta^Y$ to zero. The performance of the separating set is summarized in the next proposition.

\begin{proposition}\label{prop:finding gauss sepset}
    There exists universal constants $c,c'$ such that when $J = \lfloor c\epsilon^{-1/s}\rfloor$ the inequality
    \begin{equation*}
        \P\left(\mu_{\theta^X}(\hat S)-\mu_{\theta^Y}(\hat S) \geq c'\left(\sqrt{n\epsilon^{1/s}}\land\frac1\epsilon\right)\epsilon^2\right) \geq 1-\delta
    \end{equation*}
    holds, provided $n\gtrsim \frac{1}{c'}n_\sf{TS}(\epsilon, \delta, \cal P_\sf{G})$.
\end{proposition}

Applying Proposition \ref{prop:finding gauss sepset} and Lemma \ref{lemma:test_complexity} with the trivial bound $\tau(\hat S) \leq 1/4$ leads to the sample complexity guarantees in Table \ref{table:prior minimax} for the Gaussian sequence model class $\cal P_\sf{G}$ in all three problems $\GoF,\TS$ and $\LF$.

\section{Acknowledgements}
YH was generously supported by the Norbert Wiener postdoctoral fellowship in statistics at MIT IDSS. YP was supported in part by the National Science
Foundation under Grant No CCF-2131115. Research was sponsored by the United States Air Force Research Laboratory and the Department of the Air Force Artificial Intelligence Accelerator and was accomplished under Cooperative Agreement Number FA8750-19-2-1000. The views and conclusions contained in this document are those of the authors and should not be interpreted as representing the official policies, either expressed or implied, of the Department of the Air Force or the U.S. Government. The U.S. Government is authorized to reproduce and distribute reprints for Government purposes notwithstanding any copyright notation herein.

\bibliographystyle{alpha}
\bibliography{main}

\appendix

\section{Auxiliary Lemmas}\label{append:auxiliary}
We state some auxiliary lemmas which will be used for the proof. We begin with a simple identity for standard normal distributions. 
\begin{lemma}\label{lem:gaussian cdf expectation}
    Take $a,b \in \R$ and let $Z$ be standard normal. Then
    \begin{equation*}
        \E \Phi(aZ+b) = \Phi\left(\frac{b}{\sqrt{1+a^2}}\right). 
    \end{equation*}
\end{lemma}

\begin{proof}
    Let $Z'$ be a standard Gaussian independent of $Z$. Then
    \begin{align*}
        \E \Phi(aZ+b) = \P(aZ+b\geq Z') 
        = \P\left(\frac{Z'-aZ}{\sqrt{1+a^2}} \leq \frac{b}{\sqrt{1+a^2}}\right) 
        = \Phi\left(\frac{b}{\sqrt{1+a^2}}\right). 
    \end{align*}
\end{proof}

The following lemma is the celebrated result of Gaussian Lipschitz concentration. 
\begin{lemma}[Lipschitz concentration for Gaussians {{\cite[Theorem 5.2.1]{vershynin2018high}}}]\label{lem:gaussian lipschitz concentration}
    Let $Q$ be a $d$-dimensional standard Gaussian and let $f:\R^d \to \R$ be $\sigma$-Lipschitz. Then $f(Q)$ is sub-Gaussian with variance proxy $\sigma^2$. 
\end{lemma}

The next lemma states the Chernoff bound for Poisson random variables. 
\begin{lemma}[{\cite[Theorem 5.4]{mitzenmacher2017probability}}]\label{lem:poisson tail}
    For all $\lambda > 0$ and $x\geq0$ we have
    \begin{align*}
        \P(\poi(\lambda)-\lambda \geq x) &\leq \exp\left(-\frac{x^2}{2(\lambda+x)}\right), \\
        \P(\poi(\lambda)-\lambda \leq -x) &\leq \exp\left(-\frac{x^2}{2\lambda}\right).
    \end{align*}
\end{lemma}

The following technical lemma is helpful in establishing the Bernstein concentration in Lemma \ref{lem:bernstein master}.  
\begin{lemma}\label{lem:q(S) enough}
    Let $a\geq0, p,q \in [0,1]$ and define $\tau = p(1-p)\wedge q(1-q), \nu = p(1-p)\vee q(1-q)$. Then it always holds that 
    \begin{align*}
        a\sqrt{\frac\nu2} &\leq a\sqrt{\tau} + a^2 + |p-q|.
    \end{align*}
    In particular, if $|p-q|\ge a\sqrt{\tau} + a^2$, then
    \begin{align*}
        4|p-q| &\geq a\sqrt{\tau} + a\sqrt{\nu} + a^2. 
    \end{align*}
\end{lemma}
\begin{proof}
    After rearranging and noting that $1+2\sqrt{2}<4$, it is clear that the first inequality implies the second. Below we prove the first inequality. 

    Since the claim is invariant under the transformations $(p,q)\mapsto (q,p)$ and $(p,q)\mapsto (1-p,1-q)$, it suffices to consider the case where $p\le 1/2$ and $p(1-p)\le q(1-q)$. It further suffices to consider the case where $p\le q\le 1/2$: if not, then $p\le 1-q\le 1/2$, and the transformation $(p,q)\mapsto (p,1-q)$ keeps $(\tau,\nu)$ invariant while makes $|p-q|$ smaller. The proof is then completed by considering the following two scenarios: 
    \begin{itemize}
        \item if $p\ge q/2$, then $\nu = q(1-q)\le 2p(1-p)=2\tau$, so $a\sqrt{\nu/2}\le a\sqrt{\tau}$; 
        \item if $p\le q/2$, then $2a\sqrt{\nu}\le a^2+\nu\le a^2 + q\le a^2 + 2(q-p)$. 
    \end{itemize}
    
\end{proof}

\section{Omitted Proofs from Section \ref{sec:intro}}
\subsection{Proof of Lemma \ref{lemma:test_complexity}}
Before we prove Lemma \ref{lemma:test_complexity}, we begin with a technical lemma on the Bernstein concentration of the classifier-accuracy test \eqref{eqn:T_CA def}. 
\begin{lemma}\label{lem:bernstein master}
Suppose $A_1,\dots,A_n \stackrel{iid}{\sim} \ber(p)$ and $B_1,\dots,B_m \stackrel{iid}{\sim} \ber(q)$. Let $\tau = p(1-p)\wedge q(1-q)$ and define the averages $\bar A = \frac1n\sum_{i=1}^nA_i$ and $\bar B=\frac1m\sum_{j=1}^mB_j$. There exists a universal constant $c>0$ such that 
\begin{align*}
    \P\left(\left|\bar A-\bar B\right| \leq \frac12|p-q| - \frac12\sqrt{\frac{c\log(1/\delta) \tau}{n\land m}} - \frac12\frac{c\log(1/\delta)}{n\land m} \right) &\leq \delta, \\
    \P\left(\left|\bar A - \bar B\right| \geq 2|p-q| + 2\sqrt{\frac{c\log(1/\delta)\tau}{n\land m}} + 2\frac{c\log(1/\delta)}{n\land m}\right) &\leq \delta. 
\end{align*}
\end{lemma}

\begin{proof}
Let $\nu = p(1-p)\vee q(1-q)$. Note that the first inequality is trivially true if
\begin{equation*}
|p-q| \leq \sqrt{\frac{c\log(1/\delta) \tau}{n\land m}} + \frac{c\log(1/\delta)}{n\land m}. 
\end{equation*}
Assuming otherwise, by the second statement of Lemma \ref{lem:q(S) enough}, the first probability is upper bounded by 
\begin{align*}
     \P\left(\left|\bar A-\bar B\right| \leq |p-q| - \frac{5}{8}\sqrt{\frac{c\log(1/\delta)\tau}{n\land m}} - \frac{1}{8}\sqrt{\frac{c\log(1/\delta)\nu}{n\land m}} - \frac{5}{8}\frac{c\log(1/\delta)}{n\land m} \right). 
\end{align*}
By choosing $c$ sufficiently large (independently of $p,q,n,m,\delta$), and applying Bernstein's inequality separately to both $\bar A$ and $\bar B$, the above probability can be made smaller than $\delta$.  

For the second inequality, using the first statement of Lemma \ref{lem:q(S) enough}, it is upper bounded by
\begin{align*}
    \P\left(|\bar A-\bar B| \geq |p-q| + \sqrt{\frac{c\log(1/\delta)\tau}{n\land m}} + \frac{1}{\sqrt2}\sqrt{\frac{c\log(1/\delta) \nu}{n\land m}} + \frac{c\log(1/\delta)}{n\land m}\right).
\end{align*}
Again, taking $c$ sufficiently large (independently of $p,q,n,m,\delta$) and applying Bernstein's inequality separately to both $\bar A$ and $\bar B$, the above probability can be made smaller than $\delta$. 
\end{proof}

Now we proceed to prove Lemma \ref{lemma:test_complexity}. Using $n$ test samples $(X,Y)$ from both $p$ and $q$, consider the following classifier-accuracy test: we accept $H_0$ if
\begin{align*}
\left| \frac{1}{n}\sum_{i=1}^n \left(\one(X_i\in S) - \one(Y_i\in S) \right) \right| \le \sqrt{\frac{c\overline{\tau}\log(1/\delta)}{n}} + \frac{c\log(1/\delta)}{n}, 
\end{align*}
and reject $H_0$ otherwise. Here $c>0$ is a large absolute constant, and we note that the threshold only relies on the knowledge of $\overline{\tau}$ in addition to $(n,\delta)$.  

To analyze the type-I and type-II errors, first assume that $H_0$ holds. Since $\sep(S)=0$ under $H_0$, the second statement of Lemma \ref{lem:bernstein master} implies that we accept $H_0$ with probability at least $1-\delta/2$ if $c>0$ is large enough. If $H_1$ holds, with probability at least $1-\delta/2$, by the first statement of Lemma \ref{lem:bernstein master} we have
\begin{align*}
    \left| \frac{1}{n}\sum_{i=1}^n \left(\one(X_i\in S) - \one(Y_i\in S) \right) \right| \ge |\underline{\sep}| - \left(\sqrt{\frac{c\overline{\tau}\log(1/\delta)}{n}} + \frac{c\log(1/\delta)}{n}\right). 
\end{align*}
By the lower bound of $n$ assumed in Lemma \ref{lemma:test_complexity}, in this case we will reject $H_0$, as desired.

\subsection{Proof of Proposition \ref{prop:balanced classifier}}
\begin{lemma}\label{lem:ratio max min}
    Let $\mu$ be a non-negative measure on some space $\cal X$ and let $a,b :\cal X \to \R_+$ such that $\int a(x)\D\mu(x)>0$ and $b(x)=0$ only if $a(x)=0$. Then 
    \begin{equation*}
        \inf_{x\in\operatorname{spt}(\mu)}\left(\frac{a(x)}{b(x)}\right) \leq \frac{\int a(x)\D\mu(x)}{\int b(x)\D\mu(x)} \leq \sup_{x\in\operatorname{spt}(\mu)}\left(\frac{a(x)}{b(x)}\right). 
    \end{equation*}
\end{lemma}
\begin{proof}
    Defining $0/0=1$, we have
    \begin{align*}
        \int a(x) \D\mu(x) &= \int \frac{a(x)}{b(x)}b(x)\D\mu(x) \\
        &\leq \sup_{x\in\operatorname{spt}(\mu)} \left(\frac{a(x)}{b(x)}\right) \int b(x)\D\mu(x). 
    \end{align*}
    The other direction follows analogously. 
\end{proof}
\begin{proof}[Proof of Proposition \ref{prop:balanced classifier}]
    Let $p,q$ be the densities of $\P,\bb Q$ with respect to a common dominating measure, and let $E \eqdef \{x:p(x)> q(x)\}$ so that $\TV(\P,\bb Q)=\P(E)-\bb Q(E) > 0$. Assume without loss of generality that $\P(E) + \bb Q(E) \geq 1$. Given $t\in[0,1]$ define $E_t \eqdef \{x:\frac{p(x)-q(x)}{p(x)+q(x)} \ge t\}$, so that the map $t\mapsto \P(E_t) + \bb Q(E_t)$ is non-increasing and left-continuous. Note that $E_0=E$ while $E_1 = \varnothing$, so that $t^\star = \max\{t\in [0,1]: \P(E_t) + \bb Q(E_t) \ge 1\}$ exists. Now choose the randomized classifier $\cal C$ as follows: 
    \begin{align*}
        \cal C(x) = \begin{cases}
            0 & \text{if } x \in E_{(t^\star)^+}, \\
            1 & \text{if } x \notin E_{t^\star}, \\
            \ber(r) & \text{if } x \in E_{t^\star} -E_{(t^\star)^+}, 
        \end{cases}
    \end{align*}
    where $E_{(t^\star)^+} = \cap_{t>t^\star} E_t \subseteq E_{t^\star}$, and 
    \begin{align*}
        r := \frac{1 - \P(E_{(t^\star)^+}) - \bb Q(E_{(t^\star)^+})}{\P(E_{t^\star}) + \bb Q(E_{t^\star}) - \P(E_{(t^\star)^+}) - \bb Q(E_{(t^\star)^+})} \in [0,1]. 
    \end{align*}
    This classifier is balanced, as 
    \begin{align*}
        &\P(\cal C(X)=0) + \bb Q(\cal C(X)=0) \\
        &= \P(E_{(t^\star)^+}) + \bb Q(E_{(t^\star)^+}) + r(\P(E_{t^\star}) + \bb Q(E_{t^\star}) - \P(E_{(t^\star)^+}) - \bb Q(E_{(t^\star)^+})) \\
        &= 1. 
    \end{align*}

    For $t\in[0,1]$ define
    \begin{equation*}
        f(t) \eqdef \begin{cases} (\P(E_t)-\Q(E_t))/(\P(E_t)+\Q(E_t)) &\text{if } \P(E_t)+\Q(E_t) > 0, \\ 1 &\text{otherwise}. \end{cases}
    \end{equation*} 
Let $0 \leq t \leq s \leq 1$, we show that $f(t) \leq f(s)$. Without loss of generality assume that $f(s)<1$ and that $\P(E_s\backslash E_t)+\Q(E_s\backslash E_t) > 0$. Notice that $f(t) \leq f(s)$ if and only if
    \begin{align*}
        \frac{\int_{E_t\backslash E_s}(p(x)-q(x))\D x}{\int_{E_t\backslash E_s}(p(x)+q(x))\D x} \stackrel{!}{\leq} \frac{\int_{E_s}(p(x)-q(x))\D x}{\int_{E_s}(p(x)+q(x))\D x}.
    \end{align*}
    However, the above inequality follows from Lemma \ref{lem:ratio max min}. Thus, it holds that
    \begin{equation*}
        \frac{\P(\cal C(X)=0) - \bb Q(\cal C(X)=0)}{\P(\cal C(X)=0) + \bb Q(\cal C(X)=0)} \ge f(t^\star) \geq f(0) = \frac{\P(E) - \bb Q(E)}{\P(E) + \bb Q(E)}.
    \end{equation*}
    Plugging in $\P(\cal C(X)=0) + \bb Q(\cal C(X)=0) = 1$ and $\P(E)+\Q(E) \leq 2$ yields the result. 
    
    To show tightness, one can consider $p(x)=\one_{[0,1]}$, $q(x) = (1+\epsilon)\mathbbm{1}_{[0,1/(1+\epsilon)]}$, $\cal C(x) = \mathbbm{1}_{x\in (1/(2+\epsilon),1]}$, and let $\epsilon \to 0^+$.
\end{proof}

\section{Omitted Proofs from Section \ref{sec:learning sep sets}}
\subsection{Useful Lemmas}
Before we present the formal proofs, this section summarizes some useful lemmas on the expected value and sub-Gaussian concentration of the separation. 

\begin{lemma}\label{lem:E lower}
Let $\mu \geq \lambda \geq 0$ and $X\sim\poi(\mu),Y\sim\poi(\lambda)$. Then 
\begin{equation*}
    \P(X>Y)+ \frac12\P(X=Y) - \frac12\geq c\left(\frac{\mu-\lambda}{\sqrt{\lambda + 1}} \land 1\right)
\end{equation*}
holds, where $c>0$ is a universal constant. 
\end{lemma}
\begin{proof}
For $t \in [\lambda,\mu]$ define the function 
\begin{equation*}
    f(t) = \P(\poi(t) > Y) + \frac12 \P(\poi(t)=Y). 
\end{equation*}
Clearly $f(\lambda) = \frac12$. We have
\begin{align*}
    \frac{\d}{\d t} \P(\poi(t) > Y) = -\P(\poi(t) > Y) + \P(\poi(t) > Y-1) = \P(\poi(t) = Y).
\end{align*}
Similarly we get
\begin{align*}
    \frac{\d}{\d t} \P(\poi(t)=Y) = -\P(\poi(t)=Y)+\P(\poi(t)=Y-1).
\end{align*}
Thus, we obtain
\begin{align*}
    f'(t) &= \frac 12 \E\left[\P(\poi(t) \in \{Y-1,Y\})\right].
\end{align*}

Next we prove the following inequality: if $y$ is a non-negative integer with $|y-t|\le 8\sqrt{t}$, then 
\begin{align}\label{eq:Poisson_central_prob}
    \P(\poi(t) = y) = \Omega\left(\frac{1}{\sqrt{t+1}}\right). 
\end{align}
To prove \eqref{eq:Poisson_central_prob}, we distinguish three scenarios: 
\begin{enumerate}
    \item If $t<1/100$, then the only non-negative integer $y$ with $|y-t|\le 8\sqrt{t}$ is $y=0$. Therefore $\P(\poi(t)=y) = e^{-t} = \Omega(1)$.
    \item If $1/100\le t\le 100$, then $0\le y\le 180$. In this case, 
    \begin{align*}
        \P(\poi(t) = y) \ge \min_{1/100 \le t \le 100}\min_{0\le y\le 180} \P(\poi(t)=y) = \Omega(1).
    \end{align*}
    \item If $t>100$, then for $t-8\sqrt{t}\le y_1\le y_2\le t+8\sqrt{t}$, we have
    \begin{align*}
        \frac{\P(\poi(t)=y_1)}{\P(\poi(t)=y_2)} = t^{y_2-y_1}\frac{y_2!}{y_1!} = \prod_{y=y_1+1}^{y_2} \frac{t}{y} = \left(1\pm \cal O(t^{-1/2}) \right)^{\cal O(16\sqrt{t})} = \Theta(1). 
    \end{align*}
    In the above we have used that $|t/y-1| = \cal O(t^{-1/2})$ for all $y\in [y_1, y_2]$, and $y_2 - y_1 \le 16\sqrt{t}$. Consequently, 
    \begin{align*}
    \P(\poi(t)=y) = \Omega\left(\frac{\P(|\poi(t)-t|\le 8\sqrt{t})}{16\sqrt{t}}\right) = \Omega\left(\frac{1}{\sqrt{t}}\right), 
    \end{align*}
    where the last step is due to Chebyshev's inequality. 
\end{enumerate}

Now we apply \eqref{eq:Poisson_central_prob} to prove Lemma \ref{lem:E lower}. We first show that for non-negative integer $y$, 
\begin{align}\label{eq:event_implication}
\{ |y-\lambda|\le 2\sqrt{\lambda} \} \wedge \{ \sqrt{\lambda} \le \sqrt{t} \le \sqrt{\lambda} + 1 \} \Longrightarrow \{ |y-t|\le 8\sqrt{t} \}. 
\end{align}
In fact, if $\sqrt{\lambda}<\sqrt{2}-1$, then the LHS of \eqref{eq:event_implication} implies that $y=0$ and $t< 2$, thus \eqref{eq:event_implication} holds. If $\sqrt{\lambda}\ge \sqrt{2}-1$, then the LHS of \eqref{eq:event_implication} implies that
\begin{align*}
    |y-t| \le |y-\lambda| + (t - \lambda) \le 2\sqrt{\lambda} + (2\sqrt{\lambda} + 1) < 8\sqrt{\lambda}\le 8\sqrt{t}, 
\end{align*}
and \eqref{eq:event_implication} holds as well. Next, by \eqref{eq:Poisson_central_prob} and \eqref{eq:event_implication}, as well as Chebyshev's inequality $\P(|Y-\lambda| \leq 2\sqrt{\lambda}) \geq \frac34$, we have
\begin{align*}
    f'(t) &\geq \frac 38 \min\limits_{y\geq 0:|y-\lambda|\leq2\sqrt{\lambda}} \P(\poi(t) =y) \\
    &\ge \frac{3}{8} \one\{ \sqrt{\lambda} \le \sqrt{t} \le \sqrt{\lambda} + 1 \}\cdot \min_{|y-t|\le 8\sqrt{t}} \P(\poi(t) =y) \\
    &= \Omega\left(\frac{\one\{\sqrt{\lambda} \le \sqrt{t} \le \sqrt{\lambda} + 1 \} }{\sqrt{t+1}}\right) = \Omega\left(\frac{\one\{\sqrt{\lambda} \le \sqrt{t} \le \sqrt{\lambda} + 1 \} }{\sqrt{\lambda +1}}\right).
\end{align*}
Finally, for some absolute constant $c>0$ it holds that
\begin{align*}
f(\mu) - f(\lambda) = \int_\lambda^\mu f'(t)\d t \ge c\int_\lambda^\mu \frac{\one\{\sqrt{\lambda} \le \sqrt{t} \le \sqrt{\lambda} + 1 \} }{\sqrt{\lambda +1}} \d t \ge c\left(\frac{\mu-\lambda}{\sqrt{\lambda+1}}\wedge 1\right),
\end{align*}
which is the statement of the lemma. 
\end{proof}

\begin{lemma}\label{lem:scheffe separation variance}
For any $D\subseteq [k]$, each of $\sep(\hat S_s(D)), s\in\{>,<,1/2\}$ is sub-Gaussian with variance proxy $\sigma^2$ which can be bounded as  
\begin{align*}
    \sigma^2 \lesssim \sum_{i \in D} (p_i-q_i)^2 \land \frac{p_i+q_i}{n} = \cal O\left(\frac1n\right), 
\end{align*}
with universal hidden constants. 
\end{lemma}
\begin{proof}
 Using standard tail bounds of the Poisson distribution (Lemma \ref{lem:poisson tail}) we have for any $i \in D$ with $p_i>q_i$, 
\begin{align*}
    &\P(i \in \hat S_<(D)) \leq \P(i \not\in \hat S_{1/2}(D)) \leq \P(i \not\in \hat S_>(D)) \\
    &= \P(\poi(np_i) \leq \poi(nq_i)) \\
    &\leq \P\left(\poi(np_i) - np_i \leq -\frac12n(p_i-q_i)\right) + \P\left(\poi(nq_i) - nq_i > \frac12n(p_i-q_i)\right) \\
    &\leq 2\exp\left(-c\frac{n(p_i-q_i)^2}{p_i+q_i}\right)
\end{align*}
for some universal $c>0$. Similarly, if $i \in D$ with $p_i \leq q_i$ we get
\begin{align*}
    \P(i \in \hat S_>(D)) \leq \P(i \in \hat S_{1/2}(D)) \leq \P(i \notin \hat S_<(D)) = \P(\poi(np_i) \geq \poi(nq_i)) \leq 2\exp\left(-c\frac{n(p_i-q_i)^2}{p_i+q_i}\right).
\end{align*}

Using these estimates we turn to bounding the moment generating function of $\sep(\hat S_s)$ for $s \in \{>,<,1/2\}$. Before doing so, recall \cite[Theorem 2.1]{buldygin2013sub} that the best-possible sub-Gaussian variance proxy $\sigma^2_\sf{opt}(\mu)$ of the $\ber(\mu)$ distribution satisfies
\begin{equation*}
    \sigma^2_\sf{opt}(\mu) = \frac{\frac12 - \mu}{\log\left(\frac{1}{\mu}-1\right)}, 
\end{equation*}
where the values for $\mu \in \{0,\frac12,1\}$ should be understood as the limit of the above expression (resulting in $\sigma^2_\sf{opt} = 0,\frac14,0$ respectively). Notice also that $\mu \mapsto\sigma^2_\sf{opt}(\mu)$ is increasing on $[0,\frac12]$ and decreasing on $[\frac12,1]$, and
\begin{align*}
 \sigma^2_\sf{opt}(\mu) \le \begin{cases}
     \frac{2}{\log(2/\mu)} &\text{if } 0<\mu<1/4, \\
     1/4 & \text{if } 1/4\le \mu\le 3/4, \\
     \frac{2}{\log(2/(1-\mu))} &\text{if } 3/4<\mu<1. 
 \end{cases}
\end{align*}

Let $T \subseteq D$ denote the subset of indices given by $$T = \left\{i\in D:2\exp\left(-c\frac{n(p_i-q_i)^2}{p_i+q_i}\right) \geq \frac14\right\} = \left\{i\in D:(p_i-q_i)^2 \leq \frac{p_i+q_i}{n} \frac{\log(8)}{c}\right\}.$$ Now, for any $s\in \{>,<,1/2\}$, the sub-Gaussian variance proxy $\sigma_s^2$ of $\sep(\hat{S}_s) - \E \sep(\hat{S}_s) = \sum_{i \in D} (p_i-q_i)(\one\{i\in\hat S_s\}-\P(i\in\hat S_s)))$ is at most
\begin{align*}
\sigma_s^2 \le \sum_{i\in T} \frac{(p_i-q_i)^2}{4} + \sum_{i\in D\setminus T} (p_i-q_i)^2\cdot \frac{2(p_i+q_i)}{cn(p_i-q_i)^2} \lesssim \sum_{i \in D} (p_i-q_i)^2 \land \frac{p_i+q_i}{n}, 
\end{align*}
where the second step used the definition of $T$. In particular, since $\sum_{i\in D}(p_i+q_i)/n\le 2/n$, the above expression is always upper bounded by $\cal O(1/n)$. 
\end{proof}

\subsection{Proof of Proposition \ref{prop:S_1/2 E sep}}
By Lemma \ref{lem:E lower}, we have
\begin{align*}
    \E\sep(\hat S_{1/2}) &= \sum_{i\in[k]} \P(i\in \hat S_{1/2})(p_i-q_i) \\
    &= \sum_{i\in[k]}(\P(i\in \hat S_{1/2})-\frac12)(p_i-q_i) \\
    &\gtrsim \sum_{i\in[k]} \left(\frac{n|p_i-q_i|}{\sqrt{n (p_i\land q_i)+1}}\land1\right) |p_i-q_i| \\
    &\geq \min\limits_{G\subseteq[k]} \left\{ \sum_{i\in G} \frac{n(p_i-q_i)^2}{\sqrt{n(q_i\land p_i)+1}} + \sum_{i\not\in G} |p_i-q_i|\right\}. 
\end{align*}
Applying the Cauchy-Schwarz inequality twice, we can bound the first term above by
\begin{equation*}
    \sum_{i \in G}\frac{n(p_i-q_i)^2}{\sqrt{n(q_i+p_i)+1}} \geq \frac{n\left(\sum_{i\in G}|p_i-q_i|\right)^2}{\sum_{i\in G}\sqrt{n(q_i+p_i)+1}} \geq  \frac{n\left(\sum_{i\in G} |p_i-q_i|\right)^2}{\sqrt{2nk+k^2}}. 
\end{equation*}
Therefore, we get the lower bound
\begin{align*}
    \E\sep(\hat S_{1/2}) \gtrsim \min\limits_{0\leq\epsilon_1\leq\epsilon} \left\{\frac{n\epsilon_1^2}{\sqrt{k(n+k)}} + \epsilon-\epsilon_1\right\} = \begin{cases} {\epsilon^2\over \lambda} &\text{if } \epsilon < {\lambda \over 2} \\
    \epsilon-{\lambda \over 4} \ge {\epsilon\over 2} &\text{if } \epsilon \ge  {\lambda \over 2}
    \end{cases}  \gtrsim \epsilon^2 \left({1\over \epsilon} \wedge \sqrt{n\over k} \wedge {n\over k}\right) 
\end{align*}
where $\lambda = {\sqrt{k(n+k)} \over n} \asymp \sqrt{k\over n} \vee {k\over n}$. 
 
By Lemma \ref{lem:scheffe separation variance} we know that $\sep(\hat S_{1/2})$ is sub-Gaussian with variance proxy $\cal O(1/n)$, which implies that $|\sep(\hat S_{1/2})| \gtrsim \epsilon^2(\frac1\epsilon \land \sqrt{\frac nk} \land \frac nk)$ with probability at least $1-\delta$, provided that 
 \begin{equation*}
     \epsilon^2\left(\frac1\epsilon\land\sqrt{\frac nk}\land\frac nk\right) \gtrsim \sqrt{\frac{\log(1/\delta)}{n}}. 
 \end{equation*}
 The above rearranges to $n\gtrsim n_\TS(\epsilon, \delta, \cal P_\sf{D})$. 

 \subsection{Proof of Proposition \ref{prop:neg sep}}
 A direct computation gives
    \begin{align*}
        2\E\sep(\hat S_>) &= 2\sum_{i=1}^{3k} (p_i-q_i)\P(i\in\hat S_>) \\
        &=-\P\left(\poi\left(\frac{n}{2k}\right) > \poi\left(\frac nk\right)\right) + 1-e^{-n/(4k)} \\
        &\leq -(1-e^{-n/(2k)})e^{-n/k} + 1-e^{-n/(4k)} \\
        &= -e^{-n/k}+e^{-3n/(2k)} + 1-e^{-n/(4k)} \leq 0, 
    \end{align*}
    for $\exp(-n/(4k)) \gtrapprox 0.86$. Rearranging, this gives the sufficient condition $n/k \le 0.6$.

\subsection{Proof of Proposition \ref{prop:S_> + S_geq E lower}}
Similar to the proof of Proposition \ref{prop:S_1/2 E sep}, we have by Lemma \ref{lem:E lower} that 
\begin{align*}
    \E\sep(\hat S_{1/2}(D)) &= \sum_{i \in D} (p_i-q_i) \P(i\in\hat S_{1/2}(D)) \geq c\cal E(D) + \frac12 \{p(D)-q(D)\} \\
    -\E\sep(D\setminus\hat S_{1/2}(D)) &= \sum_{i \in D} (q_i-p_i)\P(i\not\in\hat S_{1/2}(D)) \geq c\cal E(D) + \frac12\{q(D)-p(D)\}
\end{align*}
where $c>0$ is universal and $\cal E(D) = \sum_{i \in D} \frac{n|p_i-q_i|^2}{\sqrt{n(p_i\land q_i)+1}} \land |p_i-q_i|$. Therefore, 
\begin{align}\label{eqn:E + E >}
    \E\left[\sep(\hat S_>(D))-\sep(\hat S_<(D))\right] &= \E\left[\sep(\hat S_{1/2}(D)) - \sep(D\setminus \hat S_{1/2}(D))\right] \geq 2c\cal E(D). 
\end{align}
The bound on the sub-Gaussian variance proxy follows directly from Lemma \ref{lem:scheffe separation variance}. 

\subsection{Proof of Corollary \ref{cor:sep+tau for P_Db}}
By a two-fold sample splitting, suppose that we have independent held out samples $(\tilde X, \tilde Y)$ identical in distribution to $(X,Y)$. In the sequel we will use samples $(X,Y)$ to construct two separating sets, and use samples $(\tilde X, \tilde Y)$ to make a choice between them. 

Let the sets $\hat S_> \eqdef \hat S_>([k]), \hat S_<\eqdef\hat S_<([k])$ be constructed using $X,Y$. By Proposition \ref{prop:S_1/2 E sep} and \ref{prop:S_> + S_geq E lower}, we have
\begin{align*}
    |\E\sep(\hat S_>)| \lor |\E\sep(\hat S_<)| &\gtrsim \epsilon^2\left(\frac{1}{\epsilon}\wedge \sqrt{\frac{n}{k}} \wedge \frac{n}{k}\right), \\
    \sigma^2(\hat S_>) + \sigma^2(\hat S_<) & \lesssim \sum_{i\in [k]} \frac{p_i+q_i}{n} \lesssim \frac{1}{k\vee n},
\end{align*}
where the last step have used that $p_i+q_i\lesssim 1/k$ in $\cal P_{\sf{Db}}$. Going forward, we assume that 
\begin{align*}
    \epsilon^2\left(\frac1\epsilon\land\sqrt{\frac nk}\land\frac nk\right) \gtrsim\sqrt{\frac{\log(1/\delta)}{k\lor n}}, 
\end{align*}
which rearranges to $n\gtrsim n_\GoF(\epsilon, \delta, \cal P_\sf{D})$. Consequently, this ensures that $|\sep(\hat S_>)|\lor|\sep(\hat S_<)| \gtrsim \epsilon^2\left(\frac1\epsilon\land\sqrt{\frac nk}\land\frac nk\right)$ with probability $1-\cal O(\delta)$. Moreover, as $n\gtrsim \log(1/\delta)$, with probability at least $1-\delta$ we have $\poi(n)\le 2n$ (cf. Lemma \ref{lem:poisson tail}). Under this event, one has $|\hat S_>| \vee |\hat S_<|\le 2n$, and 
\begin{align*}
    \tau(\hat S_>) \vee \tau(\hat S_<) \lesssim \frac{|\hat S_>| \vee |\hat S_<|}{k} \wedge 1 \le \frac{2n}{k}\wedge 1. 
\end{align*}

Next we make a choice between $\hat S_>$ and $\hat S_<$ based on held out samples $(\tilde X, \tilde Y)$. Let $\hat p, \hat q$ denote the empirical pmfs constructed using $\tilde X, \tilde Y$ respectively. For any set $A\subseteq[k]$ write $\widehat{\sep}(A) = \hat p(A)-\hat q(A)$. We define our final estimator to be
\begin{equation*}
    \hat S = \begin{cases}
        \hat S_> &\text{if } |\widehat{\sep}(\hat S_>)| \geq |\widehat{\sep}(\hat S_<)|, \\ 
        \hat S_< &\text{otherwise.} 
    \end{cases}
\end{equation*}
Clearly $\tau(\hat S)\le \tau(\hat S_>) \vee \tau(\hat S_<)\lesssim 1\wedge (n/k)$. To show the high-probability separation of $\hat S$, note that by Lemma \ref{lem:bernstein master}, it holds with probability at least $1-\cal O(\delta)$ that
\begin{align*}
|\sep(\hat S)| &\ge \frac{1}{2}|\widehat{\sep}(\hat S)| - \cal O\left(\sqrt{\frac{\tau(\hat S)\log(1/\delta)}{n}} + \frac{\log(1/\delta)}{n} \right) \\
&= \frac{1}{2}|\widehat{\sep}(\hat S_>)| \vee |\widehat{\sep}(\hat S_<)| - \cal O\left(\sqrt{\frac{\log(1/\delta)}{n\vee k}} + \frac{\log(1/\delta)}{n} \right) \\
&\ge \frac{1}{4}|\sep(\hat S_>)| \vee |\sep(\hat S_<)| - \cal O\left(\sqrt{\frac{\log(1/\delta)}{n\vee k}} + \frac{\log(1/\delta)}{n} \right) \\
&= \Omega\left(\epsilon^2\left(\frac1\epsilon\land\sqrt{\frac nk}\land\frac nk\right)  \right) - \cal O\left(\sqrt{\frac{\log(1/\delta)}{n\vee k}} + \frac{\log(1/\delta)}{n} \right). 
\end{align*}
Here the first term always dominates the second as long as $n\gtrsim n_\GoF(\epsilon, \delta, \cal P_\sf{D})$.

\subsection{Proof of Proposition \ref{prop:sep + tau guarantees P_D}}

Similar to the proof of Corollary \ref{cor:sep+tau for P_Db}, we apply a two-fold sample splitting to obtain $n$ independent held out samples $(\tilde X, \tilde Y)$. In the sequel we construct $2(\ell+2)$ candidate separating sets from $(X,Y)$, and make a choice among them using held out samples $(\tilde X, \tilde Y)$. 

The construction of the $2(\ell+2)$ separating sets is simple: for each $j\in \{0,1,\cdots,\ell+1\}$, we construct two sets $\hat S_>(D_j)$ and $\hat S_<(D_j)$. The following lemma summarizes some properties of these separating sets. Recall that we assume that $t = k\land(c_0 m/\log(1/\delta)) > n$ so that $\ell = \lceil \log_2(t/n)\rceil \geq 1$. 
\begin{lemma}\label{lemma:separating_sets}
Fix any $j\in \{0,1,\cdots,\ell+1\}$, and let $\epsilon_j = \sum_{i\in D_j} |p_i - q_i|$. With probability at least $1-\delta$, the following statements hold: 
\begin{enumerate}
    \item if $j=0$, then
    \begin{align*}
    \left|\sep(\hat S_>(D_0)) \right| \vee \left|\sep(\hat S_<(D_0)) \right| \gtrsim E_0 - \cal O\left(\sqrt{\frac{E_0\log(1/\delta)}{n}}\right), 
\end{align*}
    where
    \begin{align*}
        E_0 = \sum_{i\in D_0} n|p_i-q_i|^2 \wedge |p_i-q_i| \gtrsim \frac{n \epsilon_0^2}{k} =: \tilde E_0(\epsilon_0). 
    \end{align*}
    \item if $j\in [\ell]$, then
    \begin{align*}
        \left|\sep(\hat S_>(D_j)) \right| \vee \left|\sep(\hat S_<(D_j)) \right| \gtrsim E_j - \cal O\left(\sqrt{\frac{E_j\log(1/\delta)}{n}}\right),
    \end{align*}
    where
    \begin{align*}
        E_{j} = \sum_{i\in D_{j}} n|p_i-q_i|^2 \wedge |p_i-q_i| \gtrsim \frac{n \epsilon_j^2}{\sqrt{kt/2^j}}=: \tilde E_j(\epsilon_j). 
    \end{align*}
    \item if $j=\ell+1$, then
        \begin{align*}
        \left|\sep(\hat S_>(D_{\ell+1})) \right| \vee \left|\sep(\hat S_<(D_{\ell+1})) \right| \gtrsim E_{\ell+1} - \cal O\left(\sqrt{\frac{\log(1/\delta)}{n}}\right),
    \end{align*}
    where
    \begin{align*}
        E_{\ell+1} = \sum_{i\in D_{\ell+1}} \frac{n|p_i-q_i|^2}{\sqrt{nq_i}} \wedge |p_i-q_i| \gtrsim \sqrt{\frac{n}{k}}\epsilon_{\ell+1}^2 =: \tilde E_{\ell+1}(\epsilon_{\ell+1}). 
    \end{align*}
\end{enumerate}
\end{lemma}
\begin{proof}
We prove the above statements separately. 
\begin{enumerate}
    \item Case I: $j=0$. By Proposition \ref{prop:S_> + S_geq E lower}, it holds that
    \begin{align*}
    \E[\sep(\hat S_>(D_0)) - \sep(\hat S_<(D_0))] \gtrsim \sum_{i\in D_0} n|p_i-q_i|^2 \wedge |p_i-q_i| = E_0, 
    \end{align*}
    where we have used Lemma \ref{lemma:localization} that $q_i\le 2/t\le 2/n$ for all $i\in D_0$. Moreover, 
    \begin{align*}
        &\sigma^2(\sep(\hat S_>(D_0)))\vee \sigma^2(\sep(\hat S_<(D_0))) \\
        &\lesssim \sum_{i\in D_0} |p_i-q_i|^2 \wedge \frac{p_i+q_i}{n} \lesssim \sum_{i\in D_0} \frac{1}{n}\left(n|p_i-q_i|^2 \wedge |p_i-q_i|\right) = \frac{E_0}{n}, 
    \end{align*}
    where the last inequality is due to the following deterministic inequality: if $q\le 2/n$, then
    \begin{align*}
        |p-q|^2 \wedge \frac{p+q}{n} \lesssim \frac{1}{n}\left(n|p-q|^2 \wedge |p-q|\right). 
    \end{align*}
    The proof of the above deterministic inequality is based on two cases: 
    \begin{itemize}
        \item if $p\le 3/n$, then $|p-q|^2 \lesssim |p-q|^2 \wedge (|p-q|/n)$; 
        \item if $p > 3/n$, then $p+q\lesssim n|p-q|^2 \wedge |p-q|$.
    \end{itemize}
    
    Consequently, we have the first statement. For the second statement, similar to the proof of Proposition \ref{prop:S_1/2 E sep} we have
    \begin{align*}
        E_0 \ge \min_{\epsilon_0' \in [0,\epsilon_0]} \left(\frac{n(\epsilon_0')^2}{k} + \epsilon_0 - \epsilon_0'\right) \gtrsim \epsilon_0^2\left(\frac{1}{\epsilon_0}\wedge \frac{n}{k}\right) \asymp \frac{n\epsilon_0^2}{k}. 
    \end{align*}
    \item Case II: $j\in [\ell]$. By Proposition \ref{prop:S_> + S_geq E lower} and Lemma \ref{lemma:localization} we have
    \begin{align*}
        \E[\sep(\hat S_>(D_j)) - \sep(\hat S_<(D_j))] \gtrsim \sum_{i\in D_j} n(p_i-q_i)^2 \wedge |p_i-q_i| = E_j. 
    \end{align*}
    Similar to Case I, we have 
     \begin{align*}
        \sigma^2(\sep(\hat S_>(D_j)))\vee \sigma^2(\sep(\hat S_<(D_j))) \lesssim \sum_{i\in D_j} |p_i-q_i|^2 \wedge \frac{p_i+q_i}{n} \lesssim \frac{E_j}{n}, 
    \end{align*}
    and the first statement follows. 
    
    For the second statement, note that $|D_j|\le t/2^{j-1}=\cal O(\sqrt{kt/2^j})$ by Lemma \ref{lemma:localization}. Therefore, 
    \begin{align*}
    E_j \ge \min_{\epsilon_j' \in [0,\epsilon_j]} \left(\frac{n(\epsilon_j')^2}{|D_j|} + \epsilon_j - \epsilon_j'\right) \gtrsim \epsilon_j^2\left(\frac{1}{\epsilon_j}\wedge \frac{n}{\sqrt{kt/2^j}}\right) \asymp \frac{n\epsilon_j^2}{\sqrt{kt/2^j}}. 
    \end{align*}
    \item Case III: $j=\ell+1$. By Proposition \ref{prop:S_> + S_geq E lower} and Lemma \ref{lemma:localization}, we have
    \begin{align*}
        \E[\sep(\hat S_>(D_{\ell+1})) - \sep(\hat S_<(D_{\ell+1}))] \gtrsim \sum_{i\in D_{\ell+1}} \frac{n(p_i-q_i)^2}{\sqrt{nq_i}} \wedge |p_i-q_i| = E_{\ell+1}. 
    \end{align*}
    The first statement then follows from Lemma \ref{lem:scheffe separation variance}. The second statement then follows from
    \begin{align*}
        E_{\ell+1} \ge \min_{\epsilon_{\ell+1}' \in [0,\epsilon_{\ell+1}]} \left(\frac{n(\epsilon_{\ell+1}')^2}{\sqrt{nk}} + \epsilon_{\ell+1} - \epsilon_{\ell+1}'\right) \gtrsim \epsilon_{\ell+1}^2\left(\frac{1}{\epsilon_{\ell+1}}\wedge \sqrt{\frac{n}{k}}\right) \asymp \sqrt{\frac nk} \epsilon^2_{\ell+1}. 
    \end{align*}
\end{enumerate}
The proof is complete. 
\end{proof}

Based on Lemma \ref{lemma:separating_sets}, we are about to describe how we choose from the sets $\{\hat S_>(D_j), \hat S_<(D_j)\}_{j=0}^{\ell+1}$. Similar to the proof of Corollary \ref{cor:sep+tau for P_Db}, using the held out samples $(\tilde X, \tilde Y)$, we can obtain the empirical estimates $\widehat{\sep}(\hat S_s(D_j))$ for all $s\in \{>,<\}$ and $j\in \{0,1,\cdots,\ell+1\}$. With a small absolute constant $c_1>0$ and $\tilde E_j$ as defined in Lemma \ref{lemma:separating_sets}, the selection rule is as follows: if there is some $s\in \{>,<\}$ and $j\in \{0,1,\cdots,\ell+1\}$ such that
\begin{align*}
|\widehat{\sep}(\hat S_s(D_j))| \ge c_1\tilde E_j(\epsilon/(\ell+2)), 
\end{align*}
then choose $\hat S = \hat S_s(D_j)$; if there is no such pair $(s,j)$, choose an arbitrary $\hat S$. 

We first show that with probability at least $1-\cal O(k\delta)$, such a pair $(s,j)$ exists. Since $\|p-q\|_1\ge \epsilon$, there must exist some $j\in \{0,1,\cdots,\ell+1\}$ such that $\epsilon_j \ge \epsilon/(\ell+2)$. As long as
\begin{align*}
    n \ge c_2 n_\GoF(\epsilon/\ell,\delta,\cal P_{\sf D})
\end{align*}
for a large constant $c_2>0$, one can check via Lemma \ref{lemma:separating_sets} that $|\sep(\hat S_>(D_j)) | \vee |\sep(\hat S_<(D_j)) | \ge 4c_1\tilde E_j(\epsilon/(\ell+2))$ for a small enough universal constant $c_1>0$. Assuming that $n\gtrsim\log(1/\delta)$, we have $\tau(\hat S_>(D_j))\vee \tau(\hat S_<(D_j)) = \cal O(n2^j/t)$ with probability $1-\cal O(\delta)$ due to Poisson concentration (Lemma \ref{lem:poisson tail}). On this event, it holds with probability at least $1-\delta$ that (cf. Lemma \ref{lem:bernstein master})
\begin{align*}
|\widehat{\sep}(\hat S_>(D_j)) | \vee |\widehat{\sep}(\hat S_<(D_j)) | 
&\ge 2c_1\tilde E_j(\epsilon/(\ell+2))- \cal O\left(\sqrt{\frac{2^j\log(1/\delta)}{t}} + \frac{\log(1/\delta)}{n}\right), 
\end{align*}
which is at least $c_1\tilde E_j(\epsilon/(\ell+2))$ as long as  
\begin{align}\label{eq:sample_complexity}
    n \sqrt{\frac tk} \asymp n \sqrt{1\land \frac{m}{\log(1/\delta)k}} \ge c_3 n_\GoF(\epsilon/\ell,\delta,\cal P_{\sf D})
\end{align}
for some large $c_3>0$. Therefore, provided \eqref{eq:sample_complexity} holds, the desired pair $(j,s)$ exists with probability $1-\cal O(k\delta)$ due to a union bound. 

Conversely, if $|\widehat{\sep}(\hat S_s(D_j))| \ge c_1\tilde E_j(\epsilon/(\ell+2))$ holds for some $(s,j)$, the true separation $|{\sep}(\hat S_s(D_j))|$ is at least of the same order as well. Indeed, Lemma \ref{lem:bernstein master} shows that
\begin{align*}
|{\sep}(\hat S_s(D_j))| \ge \frac{1}{2}|\widehat{\sep}(\hat S_s(D_j))| - \cal O\left(\sqrt{\frac{2^j\log(1/\delta)}{t}} + \frac{\log(1/\delta)}{n}\right), 
\end{align*}
which is at least $c_1E_j(\epsilon/(\ell+2))/4$ as long as \eqref{eq:sample_complexity} holds. This completes the proof.

\subsection{Proof of Proposition \ref{prop:finding gauss sepset}}
The statement of Proposition \ref{prop:finding gauss sepset} follows immediately from the following lemma.

\begin{lemma}\label{lem:gaussian separation}
Let $\operatorname{sep}(\hat S) \eqdef \mu_{\theta^X}(\hat S) - \mu_{\theta^Y}(\hat S)$. There exist universal constants $c_i > 0,i\in[5]$ such that for $J = \lfloor c_1\epsilon^{-1/s}\rfloor$ we have
    \begin{align*}
    \E[ \operatorname{sep}(\hat S)] + \frac{c_2}{\sqrt n} &\geq  \frac{c_3\epsilon^2}{\epsilon + \sqrt{J/n}} \\
    \P\left(\left|\sep(\hat S) - \E\sep(\hat S)\right| \geq t+\frac{c_4}{\sqrt n}\right) &\leq 2\exp(-c_5 n t^2) 
    \end{align*}
    for all $t\geq0$. 
\end{lemma}

\begin{proof}
    Write $\|\cdot\|,\left\langle\cdot,\cdot\right\rangle$ for the $\ell^2$ norm/inner product restricted to the first $J$ coordinates. Notice that given $\hat\theta^X$ and $\hat\theta^Y$, $T(\theta)$ is simply a Gaussian random variable with $\E T(\theta) = \|\hat\theta^Y-\theta\|^2-\|\hat\theta^X-\theta\|^2$ and $\var(T) = 4\|\hat\theta^X-\hat\theta^Y\|^2$. Define the vectors
\begin{align*}
    U &= \{\hat\theta^X_j-\hat\theta^Y_j\}_{j=1}^J \\
    V &= \{\hat\theta^X_j+\hat\theta^Y_j\}_{j=1}^J. 
\end{align*}
Note that they are independent, jointly Gaussain with variance $2I_J/n$ and means equal to the first $J$ coordinates of $\theta^X\mp\theta^Y$ respectively. Let $\Phi$ be the cdf of the standard Gaussian and $\phi = \Phi'$ be its density. The separation can be written as
\begin{align*}
    \operatorname{sep}(\hat S) &= f(\theta^X) - f(\theta^Y), 
\end{align*}
where
\begin{equation}
    f(\theta) = \Phi\left(\frac{\|\hat{\theta}^Y-\theta\|^2-\|\hat\theta^X-\theta\|^2}{2\|\hat\theta^X-\hat\theta^Y\|}\right) = \Phi\left(-\frac12 \left\langle V ,\frac{U}{\|U\|}\right\rangle + \left\langle \theta, \frac{U}{\|U\|}\right\rangle\right). 
\end{equation}
We focus on proving the desired tail bound first. To make the dependence on the variables explicit, write $g(U,V) = f(\theta^X)-f(\theta^Y)$ for the separation. Given $U$, $V$ is a $\cal N(\theta^X+\theta^Y, 2I_j/n)$ random variable. Differentiating $g$ and using that $\phi$ is $1/\sqrt{2\pi e}$-Lipschitz we have
\begin{align*}
    \|\nabla_V g(U,V)\| &= \Big\|-\frac12 \frac{U}{\|U\|} \Bigg(\phi\left(-\frac12\left\langle V,\frac{U}{\|U\|}\right\rangle + \left\langle \theta^X, \frac{U}{\|U\|}\right\rangle\right) \\&\qquad- \phi\left(-\frac12\left\langle V,\frac{U}{\|U\|}\right\rangle + \left\langle \theta^Y, \frac{U}{\|U\|}\right\rangle\right)\Bigg)\Big\| \\
    &\leq \frac{1}{\sqrt{8\pi e}} \left|\left\langle \theta^X-\theta^Y, \frac{U}{\|U\|}\right\rangle\right| \\
    &\leq \frac{C_\sf{G}}{\sqrt{8\pi e}}. 
\end{align*}
By Lipschitz concentration of the Gaussian distribution (Lemma \ref{lem:gaussian lipschitz concentration}) we conclude that $g-\E[g|U]$ is sub-Gaussian with variance proxy $C_\sf{G}^2/(4\pi en)$. Next we study the concentration of $\E[g|U]$. To this end, note that 
\begin{equation*}
    \left.-\frac12\left\langle V, \frac{U}{\|U\|}\right\rangle + \left\langle \theta, \frac{U}{\|U\|}\right\rangle\right| U \sim \cal N\left(\left\langle \theta-\frac12(\theta^X+\theta^Y), \frac{U}{\|U\|}\right\rangle, \frac{1}{2n}\right).
\end{equation*}
Thus, using the independence of $U$ and $V$ and Lemma \ref{lem:gaussian cdf expectation} we obtain
\begin{align*}
    \E\left[g(U,V)|U\right] &= \E\left[f(\theta^X)-f(\theta^Y)|U\right] \\
    &= \Phi\left(\frac{W}{\sqrt{4+2/n}}\right) - \Phi\left(-\frac{W}{\sqrt{4+2/n}}\right), 
\end{align*}
where we write $W \eqdef \left\langle \theta^X-\theta^Y, \frac{U}{\|U\|}\right\rangle$. Let $\tilde\Phi = \Phi(\cdot/\sqrt{4+2/n})$ to ease notation. Once again by Lipschitzness of $\Phi$, we obtain for every $t\geq0$ that
\begin{align*}
    \P\left(\left|\tilde\Phi(W)-\E\tilde\Phi(W)\right| \geq t\right) &\leq \P\left(\left|\tilde\Phi(W)-\tilde\Phi(\E W)\right| \geq t - \|\tilde\Phi\|_\sf{Lip}\sqrt{\var(W)}\right) \\ 
    &\leq \P\left(\left|W-\E W\right| \geq \frac{t}{\|\tilde\Phi\|_\sf{Lip}} - \sqrt{\var(W)}\right), 
\end{align*}
and an analogous inequality can be obtained for $-W$. The last ingredient is showing that $W$ concentrates well.  
\begin{lemma}\label{lem:W subG}
    $W$ is sub-Gaussian with variance proxy $1/(2n)$. 
\end{lemma}
\begin{proof}[Proof of Lemma \ref{lem:W subG}]
To simplify notation, let $\tau = \theta^X-\theta^Y$, $\sigma^2 = 1/(2n)$ and let $Q$ be a zero-mean identity-covariance Gaussian random vector so that
\begin{equation*}
    W \stackrel{d}{=} \left\langle\tau, \frac{\tau + \sigma Q}{\|\tau + \sigma Q\|}\right\rangle. 
\end{equation*}
We have
\begin{align*}
    \left\langle\tau, \frac{\tau + \sigma Q}{\|\tau + \sigma Q\|}\right\rangle  &= \underbrace{\left\langle \frac{\tau}{\E\|\tau + \sigma Q\|}, \frac{\tau + \sigma Q}{\|\tau + \sigma Q\|} \right\rangle}_{|\cdot|\leq1 \text{ almost surely}} \underbrace{(\|\tau+\sigma Q\|-\E\|\tau+\sigma Q\|)}_{\sigma^2\text{ sub-Gaussian}} + \underbrace{\sigma \left\langle\frac{\tau}{\E\|\tau+\sigma Q\|}, Q\right\rangle}_{\sigma^2\text{ sub-Gaussian}}, 
\end{align*}
where we use that $\E \|\tau + \sigma Q\| \geq \|\tau\|$ by Jensen's inequality, and apply Lemma \ref{lem:gaussian lipschitz concentration} twice. Overall, this implies that $W$ is sub-Gaussian with variance proxy $\sigma^2=1/(2n)$ as required. 
\end{proof}

Recall that we have decomposed the separation as follows:
\begin{equation*}
    \sep(\hat S) - \E\sep(\hat S) = \underbrace{g - \E[g|U]}_{\cal O(1/n) \text{ sub-Gaussian}} + \underbrace{\tilde\Phi(W)-\tilde\Phi(-W) - \E[\tilde\Phi(W)-\tilde\Phi(-W)]}_{\cal O(1/n) \text{ sub-Gaussian tails beyond } \cal O(1/\sqrt{n})}, 
\end{equation*}
which completes the proof. 

Let us turn to calculating the expected separation. We have already seen that
\begin{equation*}
    \E\operatorname{sep}(\hat S) = \E\left[\tilde\Phi(W) - \tilde\Phi(-W)\right]. 
\end{equation*}
Again by Lipschitzness we have $|\E\tilde\Phi(W)-\tilde\Phi(\E W)| \leq \|\tilde\Phi\|_\sf{Lip} \E|W-\E W| \lesssim 1/\sqrt n$ by Lemma \ref{lem:W subG}. Thus, we see that
\begin{align*}
    \E\operatorname{sep}(\hat S) + \Omega\left(\frac{1}{\sqrt n}\right) \geq \tilde\Phi(\E W) - \tilde\Phi(-\E W), 
\end{align*}
where the implied constant is universal. To simplify notation, let $\tau = \theta^X-\theta^Y$, $\sigma^2 = 1/(2n)$ and let $Q$ be a standard normal random variable. Looking at $\E W$ we have
\begin{align*}
    \E W = \E \left\langle \tau, \frac{\tau + \sigma Q}{\|\tau + \sigma Q\|}\right\rangle = \frac1\sigma \E \left\langle\tau, \nabla_Q\|\tau+\sigma Q\|\right\rangle = \frac1\sigma\E\left[\left\langle\tau, Q\right\rangle \|\tau + \sigma Q\|\right]
\end{align*}
by Stein's identity. By the rotational invariance of the Gaussian distribution, the above is equal to 
\begin{align*}
    \E W &= \frac{\|\tau\|}{\sigma} \E \left[Q_1\sqrt{(\|\tau\| + \sigma Q_1)^2 + \dots + \sigma^2 Q_J^2}\right] \\
    &= \frac{\|\tau\|}{\sigma} \E\left[ Q_1 \sqrt{(\|\tau\| + \sigma Q_1)^2 + \dots + \sigma^2 Q_J^2} - Q_1 \sqrt{\|\tau\|^2 + \sigma^2 Q_1^2 + \dots + \sigma^2 Q_J^2}\right] \\
    &= 2\|\tau\|^2 \E\left[\frac{Q_1^2}{\sqrt{(\|\tau\| + \sigma Q_1)^2 + \dots + \sigma^2 Q_J^2} + \sqrt{\|\tau\|^2 + \sigma^2 Q_1^2 + \dots + \sigma^2 Q_J^2}}\right]. 
\end{align*}
By the Cauchy-Schwarz inequality we have
\begin{align*}
    (\E |Q_1|)^2 \lesssim \E\left[\frac{Q_1^2}{\sqrt{(\|\tau\| + \sigma Q_1)^2 + \dots + \sigma^2 Q_J^2} + \sqrt{\|\tau\|^2 + \sigma^2 Q_1^2 + \dots + \sigma^2 Q_J^2}}\right] \times (\|\tau\| + \sigma\sqrt J). 
\end{align*}
Plugging into our expression for $\E W$ this yields
\begin{equation*}
    \E W \gtrsim \frac{\|\tau\|^2}{\|\tau\| + \sigma\sqrt J}. 
\end{equation*}
To clarify notation, let us now write $\|\cdot\|_J$ for the $\ell^2$-norm restricted to the first $J$ coordinates. Taking $J = c \epsilon^{-1/s}$ it holds that
\begin{align*}
    \|\tau\|_J^2 = \|\tau\|^2 - \sum_{j>J} \tau_j^2 \geq \|\tau\|^2 - J^{-2s} \sum_{j > J} \tau_j^2 j^{2s} = \|\tau\|^2 - c^{-2s} \epsilon^2 \|\tau\|_s^2. 
\end{align*}
Since $\|\tau\|_s\lesssim1$ and $\|\tau\|\geq \epsilon$ by assumption, we see that for large enough universal constant $c$ we have $\|\tau\|_J\geq\epsilon/2$. Since the map $x\mapsto x^2/(x+c)$ is increasing for $x,c > 0$ it follows that
\begin{equation*}
    \E W \gtrsim \frac{\epsilon^2}{\epsilon + \sqrt{J/n}}
\end{equation*}
for a universal implied constant. By the inequality $\Phi(x) - \Phi(-x) \geq x/2$ for $x \in [0,1]$ we obtain
\begin{align*}
    \tilde\Phi(\E W) - \tilde\Phi(-\E W) \geq 1\land \E W/2, 
\end{align*}
which completes the proof. 
\end{proof}

\section{Lower bounds}\label{section:lower bds}

Recall the notation of Section \ref{sec:fundamental problems}. Given two hypotheses $H_0, H_1$, our aim is to lower bound the minimum achievable worst-case error. To this end, we use the following standard fact:
\begin{align}\label{eqn:lower master}
    \min\limits_{\psi}\max\limits_{i=0,1} \sup_{P \in H_i} \P_{S\sim P}(\psi(S)\neq i) &\geq  \frac12\big(1-\TV(\E P_0, \E P_1)\big), 
\end{align}
where $P_0,P_1$ are any random probability distributions with $\P(P_i \in H_i)=1$ and $\E P_i$ denote the corresponding mixtures and $\TV$ denotes the total variation distance. Hence, deriving a lower bound of order $\delta$ on the minimax error reduces to the problem of finding mixtures $\E P_i$ such that $1-\TV(\E P_0,\E P_1) =\Omega( \delta)$. To this end we utilize standard inequalities between divergences. 

\begin{lemma}[\cite{yuryyihongbook}]\label{lem:1-TV>chi}
    For any probability measures $\P,\Q$ the inequalities
    \begin{align*}
        1-\TV(\P,\Q) &\geq \frac12e^{-\KL(\P \| \Q)} \geq \frac{1}{2(1+\chi^2(\P \| \Q))}
    \end{align*}
    hold, where $\KL$ and $\chi^2$ denote the Kullback-Leibler and $\chi^2$ divergence respectively. 
\end{lemma}

Many of our lower bounds will follow from reduction to prior work. 

\subsection{Lower bounds for $\cal P_\sf{Db}$}\label{sec:P_Db lower}
In \cite{gerber2022likelihood} the authors gave the construction of distributions $p_{\eta,\epsilon}, p_0 \in \cal P_\sf{Db}(k, 2)$ (originally due to Paninski) for a mixing parameter $\eta$ such that $\TV(p_{\eta,\epsilon}, p_0) = \epsilon \asymp \sqrt{\KL(p_{\eta,\epsilon}, p_0)}$ for all $\eta$, where the implied constant is universal. They further showed that 
\begin{align}\label{eqn:P_Db TS lower chi}
    \chi^2(\E_\eta p_{\eta,\epsilon}^{\otimes n}, p_0^{\otimes n}) \leq \exp\left(c\frac{n^2\epsilon^4}{k}\right)-1
\end{align}
and 
\begin{equation}\label{eqn:P_Db LFHT lower chi}
    \chi^2\Big(\E_\eta\big[p_0^{\otimes n} \otimes p_{\epsilon,\eta}^{\otimes (n+m)}\big] \Big\| \E_\eta\big[ p_0^{\otimes n}\otimes p_{\epsilon,\eta}^{\otimes n}\otimes p_0^{\otimes m}\big]\Big) \leq \exp\left(c\frac{m(n+m)\epsilon^4}{k}\right)-1
\end{equation}
for a universal $c>0$. 
\begin{remark}
    More precisely, \eqref{eqn:P_Db LFHT lower chi} can be extracted from \cite{gerber2022likelihood} using the chain rule for $\chi^2$ (as opposed to $\KL$). 
\end{remark}

\subsubsection{Lower bound for $\TS$ and $\GoF$}
Take $P_0 = p_0^{\otimes 2n}$ and $P_1=p_{\epsilon, \eta_0}^{\otimes n}\otimes p_0^{\otimes n}$ in \eqref{eqn:lower master} for a fixed $\eta_0$. Then, by Lemma \ref{lem:1-TV>chi} and the data-processing inequality we have
\begin{align*}
    1-\TV(\E P_0, \E P_1) \geq \frac12\exp(-n\KL(p_{\epsilon,\eta} \| p_0)) \geq \frac12\exp(-cn\epsilon^2) \stackrel{!}{=} \Omega(\delta)
\end{align*}
for a universal $c>0$. This shows that $\GoF,\TS$ are impossible at total error $\delta$ unless $n \gtrsim \log(1/\delta)/\epsilon^2$, which gives the first term of our lower bound.  

For the second term, consider the random measures $P_0 = p_0^{\otimes 2n}$ and $P_1 = p_0^{\otimes n}\otimes p_{\epsilon,\eta}^{\otimes n}$ in \eqref{eqn:lower master}. Then 
using \eqref{eqn:P_Db TS lower chi} and Lemma \ref{lem:1-TV>chi} we have
\begin{align*}
    1-\TV(\E P_1,\E P_0) &\geq \frac12\frac{1}{1+\chi^2(\E P_1 \|\E P_0)} \\
    &\geq \frac12 \exp\left(-c\frac{n^2\epsilon^4}{k}\right) \stackrel{!}{=} \Omega(\delta). 
\end{align*}
Therefore, $\TS$ is impossible unless $n\gtrsim \sqrt{k\log(1/\delta)}/\epsilon^2$, which yields the second term of our lower bound. 

\subsubsection{Lower bound for $\LF$}
The necessity of $m\gtrsim\log(1/\delta)/\epsilon^2$ and $n\gtrsim \sqrt{k\log(1/\delta)}/\epsilon^2$ follows as for $\TS$ above. Taking $P_0 = p_0^{\otimes n} \otimes p_{\epsilon,\eta}^{\otimes n}\otimes p_0^{\otimes m}$ and $P_1 = p_0^{\otimes n} \otimes p_{\epsilon,\eta}^{\otimes (n+m)}$ in \eqref{eqn:lower master}, using \eqref{eqn:P_Db LFHT lower chi} and Lemma \ref{lem:1-TV>chi} we obtain the inequality
\begin{align*}
    1-\TV(\E P_0, \E P_1) &\geq \frac12 \frac{1}{1+\chi^2(\E P_1 \| \E P_0)} \\&\geq \frac12 \exp\left(-c\frac{m(m+n)\epsilon^4}{k}\right) \stackrel{!}{=} \Omega(\delta). 
\end{align*}
Therefore, $\LF$ is impossible with error $\cal O(\delta)$ unless $mn\gtrsim k\log(1/\delta)/\epsilon^4$ (note that the $m^2$-term is never active), which completes the lower bound proof.

\subsection{Lower bounds for $\cal P_\sf{H}$}
We don't provide the details because they are entirely analogous to Section \ref{sec:P_Db lower} and rely on classical constructions that can be found in \cite{gerber2022likelihood}. 

\subsection{Lower bounds for $\cal P_\sf{G}$} 
Given a vector $\eta \in \{\pm 1\}^\N$ define the measure
\begin{align*}
    \P_\eta = \bigotimes\limits_{j=1}^\infty \begin{cases}\begin{rcases} \cal N(\eta_j c_1\epsilon^{\frac{2s+1}{2s}}, 1) &\text{if  } 1 \leq j \leq c_2\epsilon^{-1/s}, \\ \cal N(0, 1) &\text{otherwise.} \end{rcases}\end{cases}
\end{align*}
Let $\eta_1,\eta_2,\dots$ be iid uniform signs in $\{\pm1\}$, and $\gamma_\eta$ be the mean vector of $\P_\eta$. Writing $\|\cdot\|_s$ for the Sobolev-norm of smoothness $s$ and $\|\cdot\|$ for the Euclidean norm, we see that for any $\eta$
\begin{align*}
    \|\gamma_\eta\|_s^2 &= \sum_{j=1}^\infty j^{2s} \gamma_{\eta j}^2 = \sum_{j=1}^{c_2\epsilon^{-1/s}} j^{2s} c_1^2 \epsilon^{\frac{2s+1}{s}} \leq c_1^2 \epsilon^{\frac{2s+1}{s}} \left(2c_2\epsilon^{-1/s}\right)^{2s+1} \asymp c_1^2 c_2^{2s+1},\\
    \|\gamma_\eta\|^2 &= \sum_{j=1}^\infty \gamma_{\eta j}^2 =  c_1^2 \epsilon^{\frac{2s+1}{s}} c_2\epsilon^{-1/s} \asymp c_1^2c_2\epsilon^2. 
\end{align*}

Then for any $C_\sf{G}>0$ we can choose $c_1,c_2$ independently of $\epsilon$ such that $\P_0, \P_\eta \in \cal P_\sf{G}(s, C_\sf{G})$ almost surely and $\|\gamma_\eta\|=10\epsilon$. Then for $\epsilon \leq 1/10$ we know that
\begin{equation*}
    \TV(\P_0,\P_\eta) = 2\Phi\left(\frac{\|\gamma_\eta\|}{2}\right) - 1 \geq \epsilon.
\end{equation*}
\subsubsection{Lower bounds for $\GoF$ and $\TS$}
Take $P_0 = \P_0^{\otimes 2n}$ and $P_1 = \P_\one^{\otimes n}\otimes \P_0^{\otimes n}$. Then 
\begin{align*}
    \KL(P_0 \| P_1) = n \KL(\P_0 \| \P_\one) = n c_2\epsilon^{-1/s} \frac{(c_1\epsilon^{\frac{2s+1}{2s}}-0)^2}{2} \asymp n\epsilon^2. 
\end{align*}
 Using Lemma \ref{lem:1-TV>chi} this gives us 
 \begin{align*}
     1-\TV(P_0,P_1) \gtrsim \exp(-\KL(P_0\|P_1)) = \exp(-\Theta(n\epsilon^2)) \stackrel{!}{=} \Omega(\delta). 
 \end{align*}
 By \eqref{eqn:lower master} we know then that $n \gtrsim \log(1/\delta)/\epsilon^2$ is necessary for both $\GoF$ and $\TS$ over $\cal P_\sf{G}$. 

To get the second term in the minimax sample complexity consider the construction $P_0 = \P_0^{\otimes 2n}$ and $P_1 = \P_\eta^{\otimes n}\otimes \P_0^{\otimes n}$ where $\eta$ is a uniformly random vector of signs. Writing $\omega = c_1\epsilon^{\frac{2s+1}{2s}}$ note that 
\begin{align*}
    \E\P_\eta^{\otimes n} &= \bigotimes_{j=1}^{c_2\epsilon^{-1/s}} \left(\frac12\cal N(\omega, 1)^{\otimes n} + \frac12\cal N(-\omega, 1)^{\otimes n}\right). 
\end{align*}
From here we can compute 
\begin{align*}
    \KL(P_0 \| \E P_1) &\asymp \epsilon^{-1/s} \KL\left(\cal N(0,1)^{\otimes n} \Big\| \frac12\cal N(\omega, 1)^{\otimes n} + \frac12\cal N(-\omega, 1)^{\otimes n}\right) \\
    &\asymp \epsilon^{-1/s} \left(\frac n2 \omega^2 - \E_{X \sim \cal N(0,I_n)} \log\cosh\left(\omega \sum_{i=1}^n X_i\right)\right) \\
    &\leq \frac{\epsilon^{-1/s}}{4} n^2 \omega^4 \asymp n^2\epsilon^{\frac{4s+1}{s}}, 
\end{align*}
where we used the inequality $\log\cosh(x) \geq \frac{x^2}{2} - \frac{x^4}{12}$ for all $x\in \R$. Thus, using Lemma \ref{lem:1-TV>chi},
\begin{align*}
    1-\TV(P_0 \| \E P_1) &\gtrsim \exp(-\KL(P_0\|\E\P_1)) \geq \exp(-\Theta(n^2\epsilon^{\frac{4s+1}{s}})) \stackrel{!}{=} \Omega(\delta). 
\end{align*}
 By \eqref{eqn:lower master} we know then that $n \gtrsim \sqrt{\log(1/\delta)}/\epsilon^{\frac{2s+1/2}{s}}$ is necessary for both $\GoF$ and $\TS$ over $\cal P_\sf{G}$.

\subsubsection{Lower bounds for $\LF$}
If $m\ge n$, from the $\GoF$ lower bound $n\gtrsim n_{\GoF}$ we conclude that $mn\gtrsim n_{\GoF}^2$, as desired. Therefore, throughout this section we assume that $m<n$. 

Let $P_0 = \P_\eta^{\otimes n}\otimes\P_0^{\otimes n}\otimes\P_\eta^m$ and $P_1=\P_\eta^{\otimes n}\otimes\P_0^{\otimes n}\otimes\P_0^{\otimes m}$, where $\eta$ is a uniformly random vector of signs. Once again, we define $\omega = c_1\epsilon^{\frac{2s+1}{2s}}$. We follow a proof similar to the cases $\cal P_\sf{Db}, \cal P_\sf{H}$ in \cite{gerber2022likelihood}. We use the data processing inequality, the chain rule and tensorization of $\chi^2$:
\begin{align*}
    \chi^2(\E P_0 \| \E P_1) &= \chi^2(\E \P_\eta^{\otimes (n+m)} \| \E \P_\eta^{\otimes n}\otimes \P_0^{\otimes m}) \\ 
    &= \left(\E_{X_1} \E_{\eta_1|X} \E_{\eta_1'|X_1} \int_{\R^m} \frac{\exp\left(-\frac12\sum_{j=1}^m \left\{(z_j-\eta_1\omega)^2 + (z_j-\eta_1'\omega)^2\right\}\right)}{(2\pi)^{m/2}\exp(-\frac12\sum_{j=1}^mz_j^2)} \D z\right)^{c_2\epsilon^{-1/s}}-1,
\end{align*}
where $X_1 \sim (\frac12\cal N(\omega,1/n)+\frac12 \cal N(-\omega, 1/n))$ and $\eta_1, \eta_1' | X_1$ are iid scalar signs from the posterior $p(\cdot | X_1)$, with joint distribution $p(\eta_1, X_1) = \phi(\sqrt{n}(X_1 - \eta_1\omega))/2$.  

The Gaussian integral above can be evaluated exactly and we obtain
\begin{align*}
    \chi^2(\E P_0\|\E P_1) &= (\E_{X_1,\eta_1,\eta_1'} \exp(\omega^2m\eta_1\eta_1'))^{c_2\epsilon^{-1/s}}-1. 
\end{align*}
Now, we can calculate
\begin{align*}
    \P(\eta_1=\eta_1') &= \E_{X_1} \frac{p(X_1|\eta_1=1)^2+p(X_1|\eta_1=-1)^2}{(p(X_1|\eta_1=1)+p(X_1|\eta_1=-1))^2} \\
    & = \frac{1}{2} + \frac{1}{4}\int \frac{(p(x_1|\eta_1=1) -p(x_1|\eta_1=-1))^2}{p(x_1|\eta_1=1)+p(x_1|\eta_1=-1)}\d x_1\\
    &\leq \frac12 + \frac{1}{16} \sum_{b\in \{\pm 1\}}\chi^2(\cal N(b\omega,1/n) \| \cal N(-b\omega,1/n)) \\
    &= \frac{1}{2} + \frac{\exp(4\omega^2n)-1}{8}.
\end{align*}
Together with $\P(\eta_1=\eta_1') \leq 1$, we have
\begin{align*}
    \E_{X_1,\eta_1,\eta_1'} \exp(\omega^2m\eta_1\eta_1') &\leq e^{-\omega^2m} + \left(\frac{1}{2}+\frac{1}{2}\wedge \frac{e^{4\omega^2n}-1}{8} \right)(e^{\omega^2m}-e^{-\omega^2m}) \\
    &= \cosh(\omega^2m) + t\sinh(\omega^2m), 
\end{align*}
with $t = 1\wedge ((e^{4\omega^2n}-1)/4)$. Distinguish into two scenarios: 
\begin{itemize}
    \item if $t=1$, then $4\omega^2n\ge 1$, and the above expression is $e^{\omega^2 m} \le e^{4\omega^4nm}$; 
    \item if $t<1$, then $\omega^2 n\le 1/2$ and $t\le 8\omega^2 n$. Since $m<n$, and $\cosh(x)\le 1+x^2, \sinh(x)\le 2x$ for all $x\in [0,1]$, the above expression is at most
    \begin{align*}
        1 + (\omega^2 m)^2 + 2t\omega^2m \le \exp(17\omega^4mn). 
    \end{align*}
\end{itemize}
Combining the above scenarios, we have 
\begin{equation*}
    \chi^2(\E P_0 \| \E P_1) \leq \exp(17\omega^4nm\cdot c_2\epsilon^{-1/s})-1.
\end{equation*}
Thus, we obtain 
\begin{align*}
    1-\TV(\E P_0, \E P_1) \gtrsim \frac{1}{1+\chi^2(\E P_0 \| \E P_1)} \geq \exp(-17\omega^4nm\cdot c_2\epsilon^{-1/s}) \stackrel{!}{=} \Omega(\delta). 
\end{align*}
This gives the desired lower bound 
\begin{equation*}
    nm  \gtrsim \frac{\log(1/\delta)}{\epsilon^{\frac{4s+1}{s}}}. 
\end{equation*}

\subsection{Lower bounds for $\cal P_\sf{D}$}
Clearly all lower bounds that apply to $\cal P_\sf{Db}$ also apply to $\cal P_\sf{D}$; in particular this gives the sample complexity lower bound for $\GoF$. In addition, lower bounds on the minimax high-probability sample complexity of $\TS$ were derived in \cite{diakonikolas2021optimal}. Hence, inspecting the claimed minimax rates, we only need to consider the problem $\LF$ in the cases $m \leq n \leq k$ and $n \leq m \leq k$. We give two separate constructions for the two cases, both inspired by classical constructions in the literature. As opposed to the i.i.d. sampling models, we will use the Poissonized models and rely on the formalism of pseudo-distributions as described in \cite{diakonikolas2021optimal}. Specifically, suppose we can construct a random vector $(p,q) \in [0,1]^2$ such that 1) $\E p = \E q = \Theta(1/k)$ and $\E |p-q| =\Theta(\epsilon/k)$; and 2) one of the following $\chi^2$ upper bounds hold for the Poisson mixture: 
\begin{equation}\label{eqn:pseudo-distr}
\begin{aligned}
\hspace*{-10cm} \chi^2(\E[\poi(np)\otimes\poi(nq)\otimes\poi(mp)] \|\E[ \poi(np)\otimes \poi(nq)\otimes\poi(mq)]) &\leq B(n,m,\epsilon,k), \\
    \chi^2(\E[\poi(nq)\otimes\poi(np)\otimes\poi(mp)] \| \E[\poi(np)\otimes \poi(nq)\otimes\poi(mp)]) &\leq B(n,m,\epsilon,k);
    \end{aligned}
\end{equation}
then $(n,m) \in \cal R_\sf{LF}(\epsilon, \delta, \cal P_\sf{D})$ requires $kB(n,m,\epsilon,k) \gtrsim \log(1/\delta)$ (essentially via Lemma \ref{lem:1-TV>chi}).

\subsubsection{Case $m \leq n \leq k$}\label{sec:m < n < k LB}
Suppose that $m \leq n \leq k/2$, and let $p,q$ be two random variables defined as
\begin{equation*}
    (p,q) = \begin{cases} (\frac1n, \frac1n) &\text{with probability } \frac nk, \\ (\frac{\epsilon}{k}, \frac{2\epsilon}{k}) &\text{with probability } \frac12(1-\frac nk), \\ (\frac{\epsilon}{k}, 0) &\text{with probability } \frac12(1-\frac nk). \end{cases}
\end{equation*}
Note that $\E[p]=\E[q]=\Theta(1/k)$ and $\E|p-q|=\Theta(\epsilon/k)$. Let $X,Y \in \R^3$ be random, whose distribution is given by
\begin{align*}
    X | (p,q) &\sim \poi(np) \otimes \poi(nq) \otimes \poi(mp), \\
    Y | (p,q) &\sim \poi(np) \otimes \poi(nq) \otimes \poi(mq). 
\end{align*}
Now, for any $(a,b,c) \in \N^3$ we have
\begin{align*}
    \P(X=(a,b,c)) &= \frac{1}{a!b!c!} \Big(\frac nk e^{-2-\frac mn} \left(\frac mn\right)^c + \frac12  (1-\frac nk) e^{-(3n+m)\epsilon/k}\left(\frac{\epsilon n}{k}\right)^a \left(\frac{2\epsilon n}{k}\right)^b \left(\frac{\epsilon m}{k}\right)^c \\ 
    &\qquad + \frac12 (1-\frac nk) e^{-(n+m)\epsilon/k} \left(\frac{\epsilon n}{k}\right)^a \one_{b=0} \left(\frac{\epsilon m}{k}\right)^c \Big).
\end{align*}
Similarly, for $Y$ we get
\begin{align*}
    \P(Y=(a,b,c)) &= \frac{1}{a!b!c!} \Big( \frac nk e^{-2-\frac mn} \left(\frac mn\right)^c + \frac12  (1-\frac nk) e^{-(3n+2m)\epsilon/k}\left(\frac{\epsilon n}{k}\right)^a \left(\frac{2\epsilon n}{k}\right)^b \left(\frac{2\epsilon m}{k}\right)^c \\ 
    &\qquad + \frac12  (1-\frac nk) e^{-n\epsilon/k} \left(\frac{\epsilon n}{k}\right)^a \one_{b=c=0}  \Big).
\end{align*}
In particular, we have
\begin{align*}
 \P(Y=(a,b,c)) = \Omega\left(\frac{1}{a!b!c!}\right)\begin{cases}
        1 &\text{if }(a,b,c) = (0,0,0), \\ \frac nk \left(\frac{m}{n}\right)^c &\text{otherwise.}
    \end{cases}
\end{align*}

Notice also that
\begin{align*}
    &\P(X=(a,b,c)) - \P(Y=(a,b,c)) \\
    &=\frac{1}{a!b!c!}\underbrace{\frac12(1-\frac nk) e^{-n\epsilon/k} }_{=\Theta(1)}\left(\frac{\epsilon}{k}\right)^{a+b+c} n^{a+b} m^c \underbrace{\left[2^be^{-(2n+m)\epsilon/k}\left(1-2^ce^{-m\epsilon/k}\right) + \one_{b=0}\left(e^{-m\epsilon/k}-\one_{c=0}\right)\right]}_{\eqdef I_{bc}}\\&= \frac{\Theta(1)}{a!b!c!} \left(\frac{\epsilon}{k}\right)^{a+b+c}n^{a+b}m^c I_{bc}, 
\end{align*}
where 
\begin{align}\label{eqn:I_bc bound}
    |I_{bc}| \lesssim \begin{cases} \frac{nm\epsilon^2}{k^2} &\text{if } b=c=0, \\ \frac{2^bm\epsilon}{k} &\text{if } b\geq1,c=0, \\
    \frac{n\epsilon}{k} &\text{if } b=0,c=1, \\
    2^{b+c} &\text{otherwise.} \end{cases}
\end{align}
We now turn to bounding the $\chi^2$-divergence between $X$ and $Y$. Using the estimates \eqref{eqn:I_bc bound}, we obtain 
\begin{align*}
    \chi^2(X\| Y) &= \sum_{(a,b,c)\in \N^3} \frac{(\P(X=(a,b,c))-\P(Y=(a,b,c)))^2}{\P(Y=(a,b,c))} \\
    &\lesssim I_{00}^2 + \left(\sum_{b=c=0,a\geq1} + \sum_{a\geq0,b+c\geq1}\right) \frac{\frac{1}{a!b!c!} \left(\frac{\epsilon}{k}\right)^{2a+2b+2c} n^{2a+2b} m^{2c} I_{bc}^2}{\frac nk \left(\frac {m}{n}\right)^c} \\
    &= I_{00}^2\left(1+ \sum_{a\geq1} \frac{1}{a!}\frac{\epsilon^{2a} n^{2a-1}}{k^{2a-1}}\right) + \left(\sum_{a\geq0} \frac{1}{a!} \frac{\epsilon^{2a} n^{2a}}{k^{2a}}\right) \sum_{b+c\geq1} \frac{1}{b!c!} \frac{\epsilon^{2b+2c}n^{2b+c-1}m^c}{k^{2b+2c-1}} I_{bc}^2 \\
    &\lesssim \frac{n^2m^2\epsilon^4}{k^4}\underbrace{\left(1 + \frac{n\epsilon^2}{k} e^{\epsilon^2n^2/k^2}\right)}_{=\Theta(1)} + \underbrace{e^{\epsilon^2 n^2/k^2}}_{=\Theta(1)} \sum_{b+c\geq1} \frac{1}{b!c!} \frac{\epsilon^{2b+2c}n^{2b+c-1}m^c}{k^{2b+2c-1}} I_{bc}^2.
\end{align*}
Focusing on the sum and decomposing it as $\sum_{b+c\geq1} = \sum_{c=0,b\geq1} + \sum_{b=0,c=1} + \sum_{b=0,c\geq2} + \sum_{b,c\geq1}$ we have the estimates
\begin{align*}
    &\sum_{b+c\geq1}\frac{1}{b!c!} \frac{\epsilon^{2b+2c}n^{2b+c-1}m^c}{k^{2b+2c-1}} I_{bc}^2 \\
    &\lesssim \sum_{c=0,b\geq1} \frac{1}{b!} \frac{\epsilon^{2b+2} n^{2b-1} 4^b m^2}{k^{2b+1}} + \frac{\epsilon^4mn^2}{k^3} + \sum_{b=0,c\geq2} \frac{1}{c!} \frac{\epsilon^{2c} n^{c-1}m^c 4^{c}}{k^{2c-1}} + \sum_{b,c\geq1} \frac{1}{b!c!} \frac{\epsilon^{2b+2c}n^{2b+c-1}m^c4^{b+c}}{k^{2b+2c-1}} \\
    &\lesssim \frac{\epsilon^4m^2n}{k^3}  + \frac{\epsilon^4mn^2}{k^3} + \frac{\epsilon^4m^2n}{k^3}  + \frac{\epsilon^4mn^2}{k^3} \lesssim \frac{\epsilon^4m n^2}{k^3}. 
\end{align*}
As $m\le k$, we obtain
\begin{equation*}
    \chi^2(X\| Y) \lesssim \frac{\epsilon^4mn^2}{k^3}. 
\end{equation*}

By \eqref{eqn:pseudo-distr} we conclude that in the regime $m \leq n \leq k$, $(n,m) \in \cal R_\sf{LF}(\epsilon, \delta, \cal P_\sf{D})$ requires $n^2m\gtrsim k^2\log(1/\delta)/\epsilon^4$, as desired.

\subsubsection{Case $n \leq m \leq k$}
This case is entirely analogous to the previous case with minor modifications. Suppose that $n \leq m \leq k/2$, and let $p,q$ be two random variables defined as
\begin{equation*}
    (p,q) = \begin{cases} (\frac1m, \frac1m) &\text{with probability } \frac mk, \\ (\frac{\epsilon}{k}, \frac{2\epsilon}{k}) &\text{with probability } \frac12(1-\frac mk), \\ (\frac{\epsilon}{k}, 0) &\text{with probability } \frac12(1-\frac mk). \end{cases}
\end{equation*}
Let $X,Y \in \R^3$ be random, whose distribution is given by
\begin{align*}
    X | (p,q) &\sim \poi(np) \otimes \poi(nq) \otimes \poi(mp), \\
    Y | (p,q) &\sim \poi(nq) \otimes \poi(np) \otimes \poi(mp). 
\end{align*}
Now, for any $(a,b,c) \in \N^3$ we have
\begin{align*}
    \P(X=(a,b,c)) &= \frac{1}{a!b!c!} \Big(\frac mk e^{-\frac{2n}{m}-1}\left(\frac nm\right)^{a+b} + \frac12(1-\frac mk)e^{-(3n+m)\epsilon/k}\left(\frac{\epsilon n}{k}\right)^a\left(\frac{2\epsilon n}{k}\right)^b \left(\frac{\epsilon m}{k}\right)^c \\&\qquad + \frac12(1-\frac mk)e^{-(n+m)\epsilon/k}\left(\frac{\epsilon n}{k}\right)^a\one_{b=0}\left(\frac{\epsilon m}{k}\right)^c\Big). 
\end{align*}
Similarly, for $Y$ we get
\begin{align*}
    \P(Y=(a,b,c)) &= \frac{1}{a!b!c!} \Big(\frac mk e^{-\frac{2n}{m}-1}\left(\frac nm\right)^{a+b} + \frac12(1-\frac mk)e^{-(3n+m)\epsilon/k}\left(\frac{2\epsilon n}{k}\right)^a\left(\frac{\epsilon n}{k}\right)^b \left(\frac{\epsilon m}{k}\right)^c \\&\qquad + \frac12(1-\frac mk)e^{-(n+m)\epsilon/k}\one_{a=0}\left(\frac{\epsilon n}{k}\right)^b\left(\frac{\epsilon m}{k}\right)^c\Big).
\end{align*}
In particular, we have
\begin{align*}
    \P(Y=(a,b,c)) = \Omega\left(\frac{1}{a!b!c!}\right)\begin{cases}
        1 &\text{if }(a,b,c) = (0,0,0), \\ \frac mk \left(\frac nm\right)^{a+b} &\text{otherwise.}
    \end{cases}
\end{align*}

Notice that
\begin{align*}
    &\P(X=(a,b,c))-\P(Y=(a,b,c)) \\
    &\qquad = \frac{1}{a!b!c!} \frac12(1-\frac mk) e^{-(n+m)\epsilon/k} \left(\frac{\epsilon}{k}\right)^{a+b+c} n^{a+b}m^c \underbrace{\left(e^{-2n\epsilon/k}(2^b-2^a) + \one_{b=0}-\one_{a=0} \right)}_{\eqdef J_{ab}} \\
    &\qquad= \frac{\Theta(1)}{a!b!c!} \left(\frac{\epsilon}{k}\right)^{a+b+c}n^{a+b}m^c J_{ab}, 
\end{align*}
where
\begin{equation}\label{eqn:J_ab est}
    |J_{ab}| \lesssim \begin{cases} 0 &\text{if } a+b=0, \\ \frac{n\epsilon}{k} &\text{if } a+b = 1, \\ 2^{a+b} &\text{if } a+b\geq2. \end{cases}
\end{equation}      
We now turn to bounding the $\chi^2$-divergence between $X$ and $Y$. We have
\begin{align*}
 \chi^2(X \| Y) &= \sum_{(a,b,c)\in\N^3} \frac{\big(\P(X=(a,b,c))-\P(Y=(a,b,c))\big)^2}{\P(Y=(a,b,c))} \\
 &\asymp \sum_{a+b+c\ge 1} \frac{\frac{1}{a!^2b!^2c!^2} \left(\frac{\epsilon}{k}\right)^{2a+2b+2c} n^{2a+2b}m^{2c}J_{ab}^2}{\frac{1}{a!b!c!} \frac mk \left(\frac nm\right)^{a+b}} \\
 &\asymp \sum_{a+b+c\ge 1} \frac{1}{a!b!c!} \frac{\epsilon^{2a+2b+2c} n^{a+b}m^{2c+a+b-1}J_{ab}^2}{k^{2a+2b+2c-1}} \\
 &= \underbrace{e^{\epsilon^2m^2/k^2}}_{\Theta(1)} \sum_{a+b\geq1} \frac{1}{a!b!} \frac{\epsilon^{2a+2b}n^{a+b}m^{a+b-1}J_{ab}^2}{k^{2a+2b-1}}, 
\end{align*}
where the last step follows from $J_{ab}=0$ if $a=b=0$. Now writing $t=a+b$ and distinguishing into cases $t=1$ and $t\ge 2$, by \eqref{eqn:J_ab est} we have
\begin{align*}
    \chi^2(X\|Y) \lesssim 
    \frac{\epsilon^4 n^3}{k^3} + \sum_{t\ge 2} \frac{2^t}{t!}\frac{\epsilon^{2t} n^t m^{t-1} 4^t}{k^{2t-1}} \lesssim \frac{\epsilon^4 n^3}{k^3} + \frac{\epsilon^4n^2m}{k^3} \lesssim \frac{\epsilon^4n^2m}{k^3}, 
\end{align*}
where the last line uses that $n \le m$. Once again, we can conclude by \eqref{eqn:pseudo-distr} that $n^2m\gtrsim\log(1/\delta)k^2/\epsilon^4$ is a lower bound for the sample complexity of $\LF$.

\end{document}